\documentclass{amsart}
\usepackage{graphicx} 
\usepackage{verbatim}
\usepackage[utf8]{inputenc}
\usepackage{xcolor}
\usepackage{float}
\usepackage{amsmath}
\usepackage{amssymb}
\usepackage{mathtools}
\usepackage{esint}
\usepackage[colorlinks=true,allcolors=blue]{hyperref}
\usepackage{varioref}
\usepackage{amsthm}
\usepackage{comment}
\usepackage{enumitem}
\usepackage{ifthen}
\usepackage[capitalize,noabbrev]{cleveref}
\usepackage{cases}
\usepackage{cancel}
\usepackage[toc,page]{appendix}
\usepackage{mathrsfs}

\newtheorem{theorem}{Theorem}[section]
\newtheorem{lemma}[theorem]{Lemma}

\theoremstyle{definition}

\newtheorem{rmk}[theorem]{Remark}
\newtheorem{defn}[theorem]{Definition}

\DeclareMathOperator{\re}{Re}	

\newcommand*{\C}{{\mathbb{C}}}     

\newcommand*{\R}{{\mathbb{R}}}     
  
\newcommand*{\inv}{^{-1}}
\newcommand*{\norm}[1]{\lVert#1\rVert}
\newcommand{\gl}{\lambda}

\usepackage{amsopn}

\newcommand{\tf}{\mathbb{H}}
\newcommand{\Lin}{\mathcal{L}}

\newcommand{\XTMD}{X}
\newcommand{\Xhd}{X_{\text{hd}}}
\newcommand{\Ahd}{A_{\text{hd}}}
\newcommand{\Bhd}{B_{\text{hd}}}
\newcommand{\iprod}[2]{\langle#1,#2\rangle}

\title[Polynomial stability of wind turbine tower models]{Polynomial stability of wind turbine tower models}

\begin{document}

\author[M.\ Fkirine]{Mohamed Fkirine}
\address[M.\ Fkirine]{Mathematics Research Centre, Tampere University, P.O.~ Box 692, 33101 Tampere, Finland}
\email{mohamed.fkirine@tuni.fi}

\author[L.~Paunonen]{Lassi Paunonen}
\address[L.~Paunonen]{Mathematics Research Centre, Tampere University, P.O.~ Box 692, 33101 Tampere, Finland}
 \email{lassi.paunonen@tuni.fi}
 \thanks{This work was supported by the Research Council of Finland Grant number 349002.}


 
\begin{abstract}
	We investigate the stabilization of mathematical models describing the structural dynamics of monopile wind turbine towers. 
	In the fore-aft plane, we show that the system becomes polynomially stable with an energy decay rate of $t^{-1}$ under static output feedback that relies on the velocity and/or angular velocity of the nacelle.
	Additionally, we prove that a tuned mass damper (TMD) in the nacelle ensures polynomial stability with the same energy decay rate, offering a viable alternative to active control.
	For the side-to-side plane, we analyze a model incorporating a hydraulic power transmission system and prove that feedback from the nacelle’s angular velocity and the generator load torque leads to polynomial stability of the system.
	\end{abstract}

\maketitle

\section{Introduction}
Consider a wind turbine tower anchored to the ocean floor, bearing at its peak the nacelle, which can weigh several hundred tons.
Because of its mass, the nacelle interacts with the flexible tower structure, creating complex dynamic behavior that requires a detailed model for accurate analysis. 
Wind turbines 
are 
subject to
 various types of vibrations due to environmental forces such as wind and waves. 
Among these, fore-aft vibrations (oscillations along the wind direction) and side-side vibrations (oscillations perpendicular to the wind direction)
 are particularly critical in evaluating the structural dynamics of the tower. These vibrations influence the turbine’s stability, performance, and lifespan. 

In this paper we investigate selected linear models that describe the dynamic behavior of wind turbine towers. 
In particular, we prove the strong asymptotic stability of the strongly continuous semigroup associated with each model, 
and derive an explicit decay rate for the energy of the classical solutions as $t\to\infty$. Our results are based on derivation of 
resolvent estimates for the generators of the strongly continuous semigroups associated with the wind turbine tower models.
This analysis can be used as a basis for control design for reducing vibrations in the
mathematical models.

Zhao and Weiss \cite{ZhaWei2011suppression,ZhaWei2011well,ZhaWei2014} used the NASA Spacecraft Control Laboratory Experiment (SCOLE) system as an analog to model the monopile wind turbine tower dynamics in both the fore-aft and side-to-side planes.
 In the side-to-side plane, the SCOLE system incorporates an ordinary differential equation (ODE) to represent the dynamics of the drivetrain, which transmits mechanical power from the rotor to the generator. 
 Originally designed to model large-scale space structures, the SCOLE system consists of a long elastic beam that is either clamped at one end and connected to a movable rigid body at the other or fully clamped at both ends, as detailed in \cite{LitMar88,LitMar88exact}. This analogy provides a valuable tool for simplifying the analysis of wind turbine tower dynamics. 
 Stability 
of the model can be enhanced through force and torque controls.
 Force control
can be implemented by adjusting
 the pitch angle to regulate aerodynamic torque and rotational speed, whereas torque control utilizes an electrically driven mass within the nacelle to indirectly influence the turbine’s speed.

The mathematical model, based on the SCOLE model, for the monopile wind turbine tower $\Sigma_{\text{sc}}$ in the fore-aft plane has the form
\begin{subequations}
\begin{align}
	&\rho(x) w_{tt}(x,t)=-(EI(x)w_{xx}(x,t))_{xx}, \quad(x,t)\in (0,1)\times [0,\infty),\label{01}\\
		&w(0,t)=0, \quad w_x(0,t)=0,\label{02}\\
		&mw_{tt}(1,t)-(EIw_{xx})_x(1,t)=u_1(t),\label{03}\\
		&Jw_{xtt}(1,t)+EI(1)w_{xx}(1,t)=u_2(t),\label{04}
\end{align}
\end{subequations}
where the subscripts $t$ and $x$ denote derivatives with respect to the time $t$ and the position $x$.
Equation \eqref{01} is an
 Euler-Bernoulli beam equation which 
models the dynamics of the flexible tower whose one end is clamped to the ground due to the boundary condition~\eqref{02}.
Equations \eqref{03}--\eqref{04} are the Newton-Euler rigid body equations that describe how the wind turbine 
 interacts with the 
nacelle.  
In particular, $(EIw_{xx})_x(1,t)$ and $-EI(1)w_{xx}(1,t)$ are the force and the torque exerted 
on the nacelle by the tower at time $t$.
The functions $u_1$ and $u_2$ are the open-loop control inputs that are applied to the nacelle. For simplicity of exposition we shall take 
the tower's height to be $1$. The parameters $EI$ and $\rho$ are the flexural rigidity function  and the mass density function  of the tower,
 while the transverse displacement of the wind tower at position $x$ and time $t$ is $w(x,t)$. We assume that $\rho$ and $EI$ belong
  to $C^4[0,1]$, and  are strictly positive. The parameters $m>0$ and $J>0$ are the mass and the moment of inertia of the nacelle, respectively.

	Exponential stability is typically the most desirable form of stability for controlled PDE systems.
	In \cite{Rao95,Guo2002}, it has been shown that exponential stabilization of $\Sigma_{\text{sc}}$ can be achieved using higher-order output feedback, such as $w_{xxt}(1,t)$ 
	 and $w_{xxxt}(1,t)$. However, implementing such feedback is challenging in practical applications. 
	Consequently, velocity $w_t(1,t)$ and angular velocity $w_{xt}(1,t)$ emerge as more feasible signals for
	 stabilizing the model $\Sigma_{\text{sc}}$. Unfortunately, as demonstrated in \cite{LitMar88,Rao95}, 
	 the system $\Sigma_{\text{sc}}$ cannot be stabilised exponentially using any feedback law of the form
	\begin{align*}
    u_i(t)=-\delta_{{i1}}w(1,t)-\delta_{i2}w_x(1,t)-\nu_{i1}w_t(1,t)-\nu_{i2}w_{xt}(1,t),
\end{align*}
where $\delta_{ij}$ and $\nu_{ij}$ are constants for $i,j=1,2$. Given 
this limitation, aiming for strong stability becomes a more practical and achievable objective. 
The strong stabilization of the uniform model (i.e., where $EI$ and $\rho$ are constant) in one specific case was 
studied
 in \cite{LitMar88}. 
The non-uniform case was addressed by Zhao and Weiss \cite{ZhaWei2011suppression}, who demonstrated that $\Sigma_{\text{sc}}$ can be strongly stabilized 
with
 either
 torque 
 or force control and by
 utilizing feedback from the velocity or the angular velocity of the nacelle.

 As the first main result of this paper we show that the system $\Sigma_{\text{sc}}$ also becomes \emph{polynomially stable} under either torque control, force control, or a combination of the two.
 More precisely, we show that the energy $E_{\text{sc}}(t)$ of every 
 classical state trajectory of the system $\Sigma_{\text{sc}}$ decays to zero at the rate $t^{-1}$ as $t\to\infty$.
 In addition, the energy of every mild state trajectory satisfies $E_{\text{sc}}(t)\to 0$ as $t\to\infty$.
 These results are presented in Theorems \ref{SCOLEdecay}, \ref{thm:torque}, and \ref{thm:force}.

As our results demonstrate, force and torque control are effective in strong and polynomial stabilization of wind turbine tower models. However, in practice they may reduce the power generation efficiency and increase the blade pitch actuator usage, which can lead to fatigue. 
One way to address these challenges is
to use
 passive structural control devices, 
 the most common being the tuned mass damper 
(TMD).
 In \cite{ZhaWei2017,TonZha2017}, the authors incorporated a TMD control system into the wind turbine tower system $\Sigma_{\text {sc}}$ and showed that 
the resulting closed-loop system is strongly stable.
    As the second main result of this paper, presented in Theorem \ref{thm:TMD}, we prove that the TMD stabilizes the monopile wind turbine model $\Sigma_{\text{sc}}$ polynomially.
More precisely, the energy of
 every classical solution of $\Sigma_{\text{TMD}}$ 
 decays to zero at the rate $t^{-1}$ as $t \to \infty$.
 
   In the side-side plane, we must consider the effect of the vibration torque transferred from the turbine blades to the nacelle through the dynamic
   mechanical transmission system. We study a model which incorporates the dynamics of a hydrostatic transmission system within the nacelle
    \cite{RajHsu2014, FarEla2018}.  The hydraulic pump converts mechanical rotational energy into hydraulic energy by 
      forcing hydraulic fluid (usually oil or another fluid) through a closed-loop system at high pressure, see Figure \ref{figure1}. This high-pressure fluid is then transported through 
      hydraulic lines to a hydraulic motor. At the receiving end, the hydraulic motor converts the energy stored in the pressurized fluid back into mechanical 
      rotational energy, which drives a generator to produce electricity. 
We note that an alternative model for the dynamic behaviour in the side-side plane involves 
 a two-mass drive train and a gearbox~\cite{ZhaWei2011suppression,ZhaWei2014}.
      \begin{figure}[htbp]
        \centering
        \label{figure1}\includegraphics[scale=0.4]{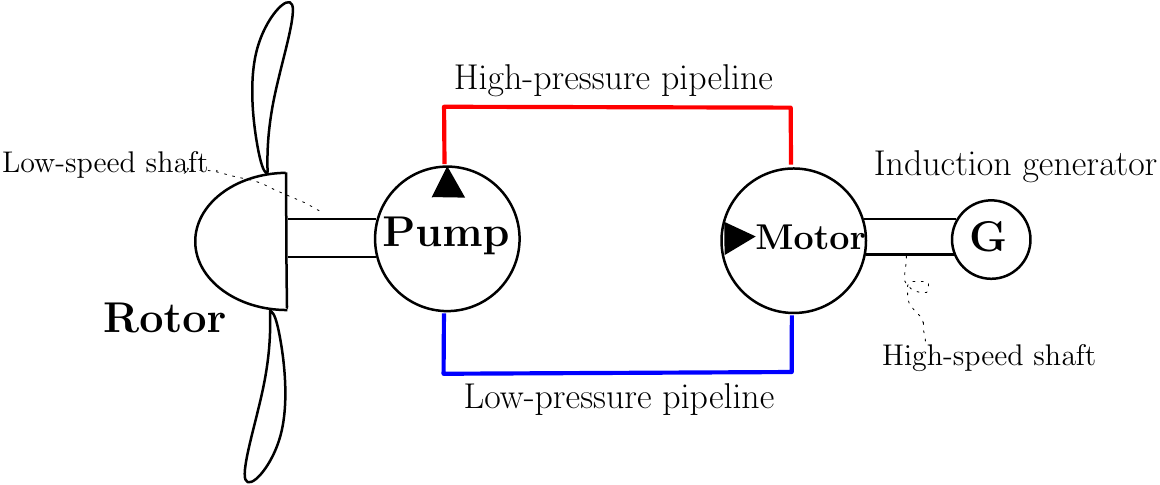}
        \caption{Schematic of a hydraulic transmission for wind turbine.}
      \end{figure}

      As our third main result we present conditions guaranteeing the strong and polynomial stability of the wind turbine tower 
system
with the dynamics
 modeling the hydraulic pump. In particular, in Theorem \ref{thm43} we show that the energy of classical solutions of
the system
decay at the rate $t^{-1}$ as $t\to\infty$.
  This result is based on a new resolvent estimate for abstract coupled systems presented in Theorem \ref{thm:AbsCouplingResEst}.

The results in this paper are related to the study of polynomial stability
of coupled systems of partial differential equations --- especially beam equations --- and ordinary differential equations.
Polynomial stability of mathematical models of a beam with a tip mass have been studied in the case of Euler--Bernoulli beams in~\cite{CheWan2007}, different variants of the Timoshenko beam model in~\cite{RivMun2015,RapVil2017,CarSil2019} and a piezoelectric beam model with magnetic effects in~\cite{AnLiu2024}.
Moreover, the polynomial stabilization of beam models using dynamic boundary control has been investigated for a rotating body–beam system in~\cite{CheWan2006} and for Rayleigh beam models in~\cite{Weh2006,MerNic2017}. 
In addition, polynomial stability of systems consisting of beam equations coupled with other one-dimensional partial differential equations have been studied, for instance, in~\cite{LiHan2018,NicHad2021,AkiGha2023}. 
The earlier studies, e.g.,~\cite{LitMar88,Rao95,Guo2002,che2003,ZhaWei2011suppression,ZhaWei2014,ZhaWei2017,TonZha2017}
 on the stability properties of the SCOLE model and its variants focus on strong and exponential stability, and 
 polynomial stability of these models has not been previously investigated.

  The rest of the paper is organized as follows. Section \ref{Sec:background} introduces the necessary background on a selected classes of infinite-dimensional linear systems and polynomial stability of semigroups. 
  In Section \ref{Sec:fore-aft} we examine the polynomial stability of the wind turbine in the fore-aft plane using feedback based on the velocity and/or angular velocity of
   the nacelle. In Section \ref{Sec:side-side} we introduce stabilization results for the hydraulic wind turbine in the side-side plane using feedback from the angular velocity of
    the nacelle and the angular velocity of the generator motor.

    \textbf{Notation.} 
Given normed spaces $X$ and $Y$, we write $\mathcal{L}(X,Y)$ for the space of bounded linear operators from $X$ to $Y$. 
We denote the domain of a linear operator $A$ by $D(A)$ and endow this space with the graph norm.  We denote the kernel of a linear operator $A$ by $\ker A$ and write $\sigma(A)$ and $\rho(A)$ for the spectrum and the resolvent set of $A$, respectively. The spectrum of $A$ is given by $\sigma(A) = \mathbb{C} \setminus \rho(A)$, with the resolvent operator defined as $R(\lambda, A) = (\lambda - A)^{-1}$. 
The inner product on a Hilbert space $X$ is denoted by $\langle \cdot, \cdot \rangle_X$. 
For $T \in \mathcal{L}(X)$, we define $\re T = \frac{1}{2}(T + T^*)$. 
If $p$ and $q$ are two real-valued quantities, we write $p \lesssim q$ to express that $p \leq Mq$ for some constant $M > 0$ that is independent of all parameters free to vary in a given situation. 
For two functions $f: [0,\infty) \to [0,\infty)$ and $g: [0,\infty) \to [0,\infty)$, we write $f(t) = o(g(t))$ if $\frac{f(t)}{g(t)} \to 0$ as $t \to \infty$, and $f(t) = O(g(t))$ if 
$f(t) \lesssim g(t)$ as $t \to \infty$.

\section{Background results on stability of abstract systems}\label{Sec:background}
In this section we present the abstract framework which we will use to analyze the stability of the wind turbine tower
in Sections~\ref{Sec:fore-aft} and~\ref{Sec:side-side}.
We also present a new result on resolvent estimates for abstract coupled systems.

Let $X_1$, $X_2$, $U$  be Hilbert spaces. Let $L_1:D(L_1)\subset X_1\to X_1$, $G_1$, $K_1:D(L_1)\subset X_1\to U$ be linear operators. Let $A_2:D(A_2)\subset X_2\to X_2$ be the generator of a $C_0$-semigroup on $X_2$,
 $B_2\in\mathcal{L}(U,X_2)$, $C_2\in\mathcal{L}(X_2,U)$ and $D_2\in\mathcal{L}(U)$. Consider 
an abstract coupled system of the form
\begin{align}
\label{eq:BCScoupled}
	\begin{cases}
	\dot z_1(t)=L_1z_1(t),&t\geq0,\\
	\dot z_2(t)=A_2z_2(t)-B_2K_1z_1(t),&t\geq0,\\
	G_1z_1(t)=C_2z_2(t)-D_2K_1z_1(t), &t\geq0,\\
	z_1(0)\in X_1, \quad z_2(0)\in X_2.
	\end{cases}
\end{align}
In our wind turbine models
 in Section~\ref{Sec:fore-aft} 
the states $z_1(t)$
and $z_2(t)$ describe the dynamics of the flexible tower and the nacelle, respectively,
 and
$G_1$ and $K_1$ are 
boundary trace operator which are used to describe the coupling between the two parts of the model.
We may recast~\eqref{eq:BCScoupled} as an abstract Cauchy problem
\begin{align}\label{pro:cauchy}
	\begin{cases}
		\dot z(t) =Az(t), &t\geq0,\\
		z(0)=z_0,
	\end{cases}
\end{align}
with
 $z(t)=(z_1(t),z_2(t))^T$ on $X=X_1\times X_2$, where $A:D(A)\subset X\to X$ is given by
\begin{subequations}
\label{oper:A}
\begin{align}
	A&=\begin{pmatrix}
		L_1&0\\
		-B_2K_1&A_2
	\end{pmatrix}\\
	D(A)&=\{(x_1,x_2)^T\in D(L_1)\times D(A_2): \; (G_1+D_2K_1)x_1=C_2x_2\}.%
\end{align}%
\end{subequations}%
We establish the well-posedness of~\eqref{pro:cauchy} under the assumptions on $(G_1,L_1,K_1)$ and $(A_2,B_2,C_2,D_2)$ listed in the following two definitions.

\begin{defn}\label{def:IPSys}
	The triple $(G_1,L_1,K_1)$ is said to be an \emph{impedance passive boundary node} on $(U,X_1,U)$ if the following hold.
	\begin{enumerate}
		\item[(i)] $G_1$, $K_1\in \mathcal{L}(D(L_1),U)$;
		\item[(ii)] the restriction of $L_1$ to $\ker G_1$ generates a $C_0$-semigroup on $X_1$;
		\item[(iii)] \sloppy  $G_1$ has a bounded right-inverse, i.e., there exists 
     $G^r\in\mathcal{L}(U,D(L_1))$
    such that $GG^r=I$;
		\item[(iv)]
$\re \langle L_1x,x\rangle_{X_1} \leq \re \langle G_1x,K_1x\rangle_{U}$ for all $x\in D(L_1)$.
	\end{enumerate}	
\end{defn}

\begin{defn}
	The
tuple
 $(A_2,B_2,C_2,D_2)$ is said to be an \emph{impedance passive linear system}  on $(U,X_2,U)$ if the following hold.
	\begin{enumerate}
		\item[(i)] $B_2\in\mathcal{L}(U,X_2)$, $C_2\in\mathcal{L}(X_2,U)$, $D_2\in\mathcal{L}(U)$;
		\item[(ii)] $A_2:D(A_2)\subset X_2\to X_2$ generates a $C_0$-semigroup on $X_2$;
		\item[(iii)] 
 $\re \langle A_2x+B_2u,x\rangle_{X_2}\leq \re\langle C_2x+D_2u,u\rangle_U$ for all $x\in D(A_2)$ and $u\in U$.
	\end{enumerate}
\end{defn}

We define the \emph{transfer function} $\tf_2: \rho(A_2)\to \mathcal L(U)$ of the impedance passive linear system $(A_2,B_2,C_2,D_2)$ by $\tf_2(\gl) = C_2R(\gl,A_2)B_2+D_2$, $\gl\in\rho(A_2)$.
The following result from \cite[Thm.~3.8]{NicPau2024} shows that 
$A$ in~\eqref{oper:A} generates a contraction semigroup and the resolvent norm $\|R(is,A)\|$ can be estimated using information on the boundary node $(G_1,L_1,K_1)$, the resolvent operator $\|R(is,A_2)\|$, and a lower bound for $\re \mathbb{H}_2(is)$.%
\begin{theorem}[{\cite[Thm.~3.8 \& Rem.~3.10]{NicPau2024}}]
\label{thm1}
	\sloppy Let $(G_1,L_1,K_1)$ be an impedance passive boundary node on $(U,X_1,U)$, let  $(A_2,B_2,C_2,D_2)$ be an impedance passive linear system on $(U,X_2,U)$.
	Let $A_0$ denote the restriction of $L_1$ to $\ker (G_1+K_1)$, and suppose there exists a non-empty set $E\subset\{s\in\mathbb{R}: \; is\in\rho(A_0)\cap \rho(A_2)\}$. 
Moreover, assume that there exists a function $\eta:E\to(0,\infty)$ such that 
$\re \mathbb{H}_2(is)\geq \eta(s)I$ for all $ s\in E$.
	Then the operator $A$ in \eqref{oper:A} generates a contraction semigroup on $X$, $iE\subset \rho(A)$ and 
	\begin{align*}
\|R(is,A)\|\lesssim \frac{(1+\|R(is,A_0\|)(1+\|R(is,A_2)\|^2)}{\eta(s)}, \quad s\in E.
	\end{align*}
\end{theorem}

The model we study in Section~\ref{Sec:side-side} can be represented as a coupled system consisting of two impedance passive linear systems. 
Let $(A_1,B_1,C_1,D_1)$ be an impedance passive linear system on $(U,X_1,U)$,
assume 
that $I+D_1D_2$ is boundedly invertible and 
denote $Q_1=(I+D_1D_2)^{-1}$ and $Q_2=(I+D_2D_1)^{-1}$. We consider 
the 
 coupled
 system
\begin{align*}
	\begin{cases}
		\dot{z}_1(t)=(A_1-B_1D_2Q_1C_1)z_1(t)+B_1Q_1C_2z_2(t), &t\geq0,\\
		\dot{z}_2(t)=(A_2-B_2Q_1D_1C_2)z_2(t)-B_2Q_1C_1z_1(t),&t\geq0,\\
		z_1(0)\in X_1, \quad z_2(0)\in X_2.
	\end{cases}
\end{align*}
We may reformulate this coupled system as an abstract Cauchy problem  \eqref{pro:cauchy}
with state $z(t)=(z_1(t),z_2(t))^T$
 on $X=X_1\times X_2$, where $A:D(A)\subset X\to X$ is defined by
\begin{align}\label{oper:Atilde}
	{A}=\begin{pmatrix}
		A_1-B_1D_2Q_1C_1&B_1Q_1C_2\\
		-B_2Q_1C_1&A_2-B_2Q_1D_1C_2
	\end{pmatrix}, \qquad D(A) = D(A_1)\times D(A_2).
\end{align}

If $K\in \mathcal{L}(U)$ is such that $K\geq cI$ for some $c>0$, then
Lemma~\ref{lem:AbsIPestimates} below shows that $I+D_1K$ is boundedly invertible and the operator
\begin{align*}
	A_K:=A_1-B_1K(I+D_1K)^{-1} C_1, \qquad D(A_K)=D(A_1).	
\end{align*}
generates a contraction semigroup on $X_1$.
The following result shows that the norm $\|R(is, A)\|$ of the resolvent operator of $A$ in \eqref{oper:Atilde}
 can be estimated 
using the resolvent norms $\|R(is, A_K)\|$ and $\norm{R(is,A_2)}$
 and a lower bound for $\re\, \mathbb{H}_2(is)$. 
We note that this result is a variation of~\cite[Thm.~3.8]{NicPau2024}.

\begin{theorem}\label{thm:AbsCouplingResEst}
	Let $(A_1,B_1,C_1,D_1)$ and $(A_2,B_2,C_2,D_2)$ be two linear impedance passive linear systems on the Hilbert spaces $(U,X_1,U)$ and $(U,X_2,U)$, assume that $I+D_1D_2$ is boundedly invertible,
and assume that $K\in \mathcal L(U)$ is such that $\re K\ge cI$ for some $c>0$.
	Moreover, assume that there exists 
	a set $E\subset \mathbb R$ and 
	a function $\eta:E\to (0,\infty)$ such that
	 $iE\subset \rho(A_K)\cap \rho(A_2)$ and 
$\re \mathbb{H}_2(is)\geq \eta(s)I$ for all $ s\in E$.
	Then the operator ${A}$ in \eqref{oper:Atilde} generates a contraction semigroup on $X$, $iE\subset \rho(A)$ 
and
	\begin{align}\label{eq:AbsResBound}
		\|R(is,{A})\|\lesssim \frac{(1+\|R(is,A_K)\|)(1+\|R(is,A_2)\|^2)}{\eta(s)}, \qquad s\in E.
	\end{align}
	\end{theorem}
	The proof of Theorem~\ref{thm:AbsCouplingResEst} is based on the following lemma.
\begin{lemma}
\label{lem:AbsIPestimates}
Let $(A,B,C,D)$ be an impedance passive linear system on $(U,X,U)$  and let $Q\in \mathcal{L}(U)$ be such that $\re Q\geq cI$ for some $c>0$.
Then $I+DQ$ is boundedly invertible and if we define
\begin{align*}
	A_Q&= A-BQ(I+DQ)^{-1}C, 
&&B_Q=B(I+QD)^{-1}, \\
C_Q&= (I+DQ)^{-1} C, 
&&D_Q=(I+DQ)^{-1} D,
\end{align*}
then $(A_Q,B_Q,C_Q,D_Q)$ is an impedance passive linear system on $(U,X,U)$.
Moreover, for all $\lambda\in \rho(A_Q)\cap \overline{\C_+}$ we have
\begin{align*}
	&\|R(is,A_Q) B_Q \|^2 
\leq c^{-1}\|R(is,A_Q)\|,\\
&\|C_Q R(is,A_Q)\|^2
\leq c^{-1}\|R(is,A_Q)\|,\\
&\|C_Q R(is,A_Q)B_Q+D_Q\|
\leq c^{-1}.
\end{align*}
\end{lemma}

\begin{proof}
	The assumption $\re Q\geq cI$ implies that $Q$ is boundedly invertible, and since $I+DQ=(Q^{-1} +D)Q$, where $\re(Q^{-1} +D)\geq \re Q^{-1}\geq c \|Q\|^{-2}I$ by~\cite[Lem.~A.1(a)]{Pau19}, we have that $I+DQ$ is boundedly invertible.
The operator $A_Q$ generates a semigroup on $X_2$ and $B_Q$, $C_Q$, and $D_Q$ are bounded.
To show that $(A_Q,B_Q,C_Q,D_Q)$ satisfies condition (iii) of Definition~\ref{def:IPSys}, denote
$T_1 = (I+D Q )\inv$, $T_2 = (I+ Q D)\inv$.
For all $x\in D(A_Q)$ and $u\in U$ the impedance passivity of $(A,B,C,D)$ 
and $\re Q\geq 0$ imply that
\begin{align*}
	\re \langle A_Q x + B_Q u,x\rangle
	  &=\re \langle A x + BT_2(- Q C x + u),x\rangle\\
	  &\hspace{-0.6em}\leq 
	  \re \langle C x + DT_2(- Q C x + u),T_2(- Q C x + u)\rangle\\
	   &\hspace{-0.6em}=\re  \langle T_1C x + T_1Du,u\rangle
	   +\re  \langle T_1C x + T_1Du,-T_2 Q C x +(T_2-I)u\rangle\\
	   &\hspace{-0.6em}=\re \langle C_Q x + D_Q  u,u\rangle
	   -\re \langle T_1C x + T_1Du,Q(T_1 C x + T_1Du)\rangle\\
	 &\hspace{-0.6em}  \leq \re  \langle C_Q x + D_Q  u , u\rangle .
\end{align*}
This 
completes the proof
 that $(A_Q,B_Q,C_Q,D_Q)$ is an impedance passive linear system.
	Now let $\gl\in\rho(A_Q)\cap\overline{\C_+}$ and $x\in X$.
If we denote $z=R(is,A_Q) x\in D(A_Q)$, then 
	 the impedance passivity of $(A,B,C,D)$ implies
\begin{align*}
	\re  \langle x,z\rangle
	&= \re \langle (is-A_Q)z,z\rangle
	= -\re \langle Az+B(-Q(I+DQ)\inv Cz),z\rangle\\
	&\geq -\re  \langle Cz+D(-Q(I+DQ)\inv Cz),-Q(I+DQ)\inv Cz\rangle\\
	&=\re  \langle Q C_Qz, C_Qz\rangle
	\geq c \|C_Q z\|^2.
\end{align*}
Since
 $C_Qz = C_QR(is,A_Q)x$ and $\re \langle x,z\rangle\leq \norm{x}\norm{z}=\norm{x}\norm{R(is,A_Q) x}$,
 the proof of the second estimate 
$\norm{C_QR(is,A_Q)  }^2 \leq c\inv\norm{R(is,A_Q)}$ is complete.

We note that  
$(A^\ast,C^\ast,B^\ast,D^\ast)$
 is also an impedance passive linear system and $\re Q^\ast = \re Q\geq cI$. Therefore, the first estimate follows from applying the second estimate to the system obtained from
$(A^\ast,C^\ast,B^\ast,D^\ast)$
 with feedback $u=-Q^\ast y$.

Finally, to prove the last estimate we let $\gl\in\C_+$ and $u\in U$ and 
denote $\tf_Q(\gl)=C_QR(is,A_Q)B_Q+D_Q$ and $\tf(\lambda)=CR(\lambda, A)B+D$.
 Since $\re Q\geq cI$ and $\re \tf(\gl)\ge 0$, we have that 
$I+Q\tf(\gl)$
 is boundedly invertible and $\tf_Q(\gl)=\tf(\gl)(I+Q\tf(\gl))\inv$.
Using this 
and $\re \tf(\lambda)\geq 0$ we can estimate
\begin{align*}
	c \|\tf_Q(\lambda)u\|^2 &\leq \re\langle Q\tf_Q(\lambda)u,\tf_Q(\lambda)u\rangle= \re\langle Q\tf(\lambda)(I+Q\tf(\lambda))\inv u,\tf_Q(\lambda)u\rangle  \\
&= \re\langle u,\tf_Q(\lambda)u\rangle 
- \re\langle (I+Q\tf(\lambda))\inv u,\tf(\lambda)(I+Q\tf(\lambda))\inv u\rangle \\
&\leq \|\tf_Q(\lambda)u\|\|u\|.
\end{align*}
Thus $\|\tf_Q(\lambda)\|\leq c\inv$ for all $\lambda\in\C_+$ and, by continuity, this estimate also holds for $\lambda\in\rho(A_Q)\cap \overline{\C_+}$.
\end{proof}
	\begin{proof}[Proof of Theorem~\textup{\ref{thm:AbsCouplingResEst}}]
		Let $E$ and $\eta$ be as in the claim and fix $s\in E$.
		Since $\re \mathbb{H}_2(is)\geq \eta(s)I$ with $\eta(s)>0$ by assumption, applying Lemma~\ref{lem:AbsIPestimates} to the system $(A_1,B_1,C_1,D_1)$ with $Q=\mathbb{H}_2(is)$ shows that $I+D_1\mathbb{H}_2(is)$ (and hence also $I+\mathbb{H}_2(is)D_1$) is boundedly invertible.
We define
\begin{gather*}
    A_{is}=A_1-B_1\mathbb{H}_2(is)(I+D_1\mathbb{H}_2(is))\inv C_1, \quad D(A_{is})=D(A),\\
		B_{is}= B_1(I+\mathbb{H}_2(is)D_1)\inv, \; C_{is}= (I+D_1\mathbb{H}_2(is))\inv C_1,\; 	D_{is}= (I+D_1\mathbb{H}_2(is))\inv D_1
\end{gather*}
and $\mathbb H_{is}(is)=C_{is} R(is,A_{is}) B_{is} +D_{is}$.
Lemma~\ref{lem:AbsIPestimates} shows that
$(A_{is},B_{is},C_{is},D_{is})$
is an impedance passive linear system and
if $is\in\rho(A_{is})$, then
		\begin{subequations}\label{eq:SchurAddEstimates}
			\begin{align}
				\|R(is,A_{is}) B_{is}\|^2 
		&\leq \eta(s)\inv\|R(is,A_{is})\|\\
		\|C_{is}R(is,A_{is})\|^2 
		&\leq \eta(s)\inv\|R(is,A_{is})\|\\
		\|\tf_{is} (is)\|
		&\leq \eta(s)\inv.
			\end{align}
		\end{subequations}
		Since $A_{is}$ is a generator of a strongly continuous contraction semigroup, the Mean Ergodic Theorem~\cite[Prop. 4.3.1 \& Cor. 4.3.5]{AreBat11book} implies that $is-A_{is}$ is boundedly invertible if and only if
		 it is lower bounded, i.e.,
		 $\|(is-A_{is})x\|\geq c_s \|x\|$ for some $c_s>0$ and for all $x\in D(A_{is})$. We will now show that if $is-A_{is}$ is boundedly invertible,
 then $is\in\rho(A)$.
		If we denote $A_2^{cl}=A_2-B_2(I+D_1D_2)\inv D_1C_2$,
		then it is straightforward to verify that
\begin{align*}
	is-A 
	= \begin{pmatrix}I&-B_{is}C_2R({is},A_2)\\0&I\end{pmatrix} 
	\begin{pmatrix}is-A_{is}&0\\0&is-A_2^{cl}\end{pmatrix}
	\begin{pmatrix}I&0\\R({is},A_2)B_2 C_{is}&I\end{pmatrix}.
\end{align*}
We also note that since $is\in\rho(A_2)$ and since $I+D_1\mathbb{H}_2(is)$ is invertible, 
a direct computation shows that $R({is},A_2)-R({is},A_2)B_2 D_{is}C_2 R({is},A_2)$ is a bounded inverse of $is-A_2^{cl}$, and thus $is\in \rho(A_2^{cl})$.
These properties imply that if $is\in \rho(A_{is})$, then $is-A$ has a bounded inverse given by
\begin{align*}
	\begin{pmatrix}R(is,A_{is}) &R(is,A_{is})B_{is} C_2 R({is},A_2)\\ -R({is},A_2)B_2 C_{is} R(is,A_{is})&R({is},A_2)- R({is},A_2)B_2  \mathbb{H}_{is}(is) C_2 R({is},A_2)\end{pmatrix}.
\end{align*}
Noting that $\eta(s)\leq \norm{P_2(is)}\lesssim 1+\norm{R(is,A_2)}$ the estimates~\eqref{eq:SchurAddEstimates} imply that
		\begin{align}\label{eq:AbsCoupGenResEst}
			\|R(is,A)\|
		&\lesssim \|R(is,A_{is})\|
+ \frac{\norm{R(is,A_2)}\norm{R(is,A_{is})}^{1/2}}{\eta(s)^{1/2}}
 +  \frac{1+\|R(is,A_2)\|^2}{\eta(s)}.
		\end{align}
Here and in the remaining part of the proof the implicit constants in the estimates are independent of the choice of $s\in E$.
		
		We will now show that $is-A_{is}$ is lower bounded and derive an estimate for $\|R(is,A_{is})\|$. To this end,
let $x\in D(A_{is})$ and denote $y=(is-A_{is}) x$.
Using 
 $A_K = A_1-B_1K(I+D_1K)\inv C_1 $ and denoting $B_K = B_1(I+KD_1)\inv$
		we can write
\begin{align*}
	y &= (is-A_1)x+B_1\mathbb{H}_2(is)(I+D_1\mathbb{H}_2(is))\inv C_1 x\\
		&= (is-A_K)x+B_K [\mathbb{H}_2(is) - K](I+D_1\mathbb{H}_2(is))\inv C_1 x.
\end{align*}
		Since $s\in E$, we have $is\in\rho(A_K)$ by assumption, and thus 
	\begin{align}\label{eq:AbsResProofx}
		x&= R(is,A_K) y- R(is,A_K)B_K [\mathbb{H}_2(is) - K] C_{is} x.
	\end{align}
		A direct estimate using the impedance passivity of $(A_1,B_1,C_1,D_1)$ and $\re \mathbb{H}_2(is)\geq \eta(s)I$ shows that  
\begin{align*}
	\re \langle y,x\rangle
		&=  \re \langle (is-A_{is})x,x\rangle
		=  -\re \langle A_1x+B_1(-\mathbb{H}_2(is)C_{is} x),x\rangle\\
		&\geq  -\re \langle C_1x+D_1(-\mathbb{H}_2(is)(I+D_1\mathbb{H}_2(is))\inv C_1 x),-\mathbb{H}_2(is) C_{is} x\rangle\\
		&=  \re \langle C_{is} x,\mathbb{H}_2(is)C_{is} x\rangle
		= \eta(s) \|C_{is} x\|^2.
\end{align*}
		This in particular implies $\|C_{is}x\|^2\leq \eta(s)\inv \|x\|\|y\|$.
		Using~\eqref{eq:AbsResProofx}, the
		above estimates, and $\|R(is,A_K) B_K\|^2\lesssim \|R(is,A_K)\|$ from 
		Lemma~\ref{lem:AbsIPestimates} shows that
\begin{align*}
	\|x\|
		\leq \|R(is,A_K)\| \|y\|
		+ \|R(is,A_K)\|^{1/2}(\norm{K}+\|\mathbb{H}_2(is)\|) \left( \frac{\|x\|\|y\|}{\eta(s)} \right)^{1/2}.
\end{align*}
Using
 Young's inequality leads to the estimate
\begin{align*}
	\|x\|\lesssim \frac{(1+\|R(is,A_K)\|)(1+\|\mathbb{H}_2(is)\|^2)}{\eta(s)}\|y\|.
\end{align*}
		Since $x\in D(A_{is})$ was arbitrary, we have that $is-A_{is}$ is lower bounded, and therefore boundedly invertible. Moreover, its inverse satisfies 
		\begin{align*}
			\|R(is,A_{is})\|\lesssim \frac{(1+\|R(is,A_K)\|)(1+\|\mathbb{H}_2(is)\|^2)}{\eta(s)}.
		\end{align*}
		Finally, using the estimate~\eqref{eq:AbsCoupGenResEst}
		and the fact that $\|\mathbb{H}_2(is)\|\leq \|C_2\|\|R(is,A_2)\|\|B_2\|+\|D_2\|$ leads to the desired estimate for $\|R(is,A)\|$.
	 \end{proof}
We conclude this section by recalling the following theorem from \cite[Thm. 2.4]{BorTom2010}. This classical result allows us to deduce a rational decay rate for the orbits of a $C_0$-semigroup based on resolvent estimates for its generator.

\begin{theorem}\label{thm:BorTom2010}
	Let $(\mathbb{T}(t))_{t\geq0}$ be a bounded $C_0$-semigroup on a Hilbert space $X$ with generator $A$ such that $i\mathbb{R}\subset \rho(A)$. Then for a fixed $\alpha>0$
	the following conditions are equivalent:
	\begin{enumerate}
		\item[(a)] $\|R(is,A)\|=O(|s|^\alpha)$ as $ |s|\to\infty$,
		\item[(b)] $\|\mathbb{T}(t)x\|=o(t^{-1/\alpha})$ as $   t\to\infty$ for all $ x\in D(A)$. 
	\end{enumerate}
\end{theorem}
If $A: D(A) \subset X \to X$ is the generator of a $C_0$-semigroup $(\mathbb{T}(t))_{t \geq 0}$ on a Hilbert space $X$, then the orbits of this semigroup are the solutions of the abstract Cauchy problem \eqref{pro:cauchy}.
In particular, 
item (b) in Theorem~\ref{thm:BorTom2010} is equivalent to $\|z(t)\| = o(t^{-1/\alpha})$ as $t\to\infty$ for every classical solution of \eqref{pro:cauchy}.

\section{Polynomial stability of the wind turbine tower model in the fore-aft plane}\label{Sec:fore-aft}
In this section we study the stability of the wind turbine tower $\Sigma_{\text{sc}}$. 
We recall that the mathematical model of $\Sigma_{\text{sc}}$ has the form
\begin{align}\label{sigmasc}
		\begin{cases}
	\rho(x) w_{tt}(x,t)=-(EI(x)w_{xx}(x,t))_{xx}, \quad(x,t)\in (0,1)\times [0,\infty),\\
		w(0,t)=0, \quad w_x(0,t)=0,\\
		mw_{tt}(1,t)-(EIw_{xx})_x(1,t)=u_1(t),\\
		Jw_{xtt}(1,t)+EI(1)w_{xx}(1,t)=u_2(t),
	\end{cases}
	\end{align}
where $\rho$ and $EI$ belong to $C^4[0,1]$ and are strictly positive. The constants $m > 0$ and $J > 0$ are positive.
The natural state of $\Sigma_{\text{sc}}$  at  time $t$ is
\begin{align}\label{statesigma}
	z(t)= (z_1(t),z_2(t),z_3(t),z_4(t))^T
= (w(\cdot,t), w_t(\cdot,t), w_t(1,t), w_{xt}(1,t))^T,
\end{align}
where $z_1$ and $z_2$ are the transverse displacement and the transverse velocity of the tower, respectively and where $z_3$ and $z_4$ are the velocity and the angular velocity of the nacelle, respectively.
The energy state space of the solutions of $\Sigma_{\text{sc}}$ is 
	$X_{\text{sc}}=H^2_l(0,1)\times L^2(0,1)\times \mathbb{C}^2$, where $H^2_l(0,1)=\{f\in H^2(0,1): \; f(0)=f'(0)=0\}$.
We equip the Hilbert space $X_{\text{sc}}$ with the norm defined by
\begin{align}
\label{eq:Xscnorm}
	\|(f,g,\xi,\eta)\|_{X_\text{sc}}^2=\int_0^1\left[EI(x)|f''(x)|^2+\rho(x)|g(x)|^2\right]dx+m|\xi|^2+J|\eta|^2.
\end{align}
The wind turbine model $\Sigma_{\text{sc}}$ can be rewritten in the form 
\begin{subequations}
\label{dessigma}
\begin{align}
	\dot z(t)&=A_{\text{sc}}z(t)+B_\text{sc}v(t), \qquad z(0)=z_0\\
	y(t)&= C_{\text{sc}}z(t)
\label{dessigma_output}
\end{align}
\end{subequations}
on $X_{\text{sc}}$.
Here $B_{\text{sc}}$ is the control operator and $v$ is the control input which is chosen based on the type of control applied to the nacelle. 
Similarly, the measurement $y$
which will be used in the feedback stabilization 
 and the corresponding $C_{\text{sc}}$ 
 are determined by the control scenario.
The main operator
$A_{\text{sc}}: D(A_{\text{sc}})\subset X_{\text{sc}}\to X_{\text{sc}}$ is defined by
\begin{subequations}
\label{scolegenerator}
\begin{align} 
 &A_{\text{sc}}\begin{pmatrix}
			f\\g\\\xi\\\eta
		\end{pmatrix}	
			=\begin{pmatrix}
			g\\-\frac{1}{\rho(x)}(EIf'')''(x)\\\frac{1}{m}(EIf'')'(1)\\-\frac{1}{J}(EIf'')(1)
		\end{pmatrix},\\
&D(A_{sc})
=\left\{(f,g,\xi,\eta)\in \left(H^2_l(0,1)\cap H^4(0,1)\right)\times H^2_l(0,1)\times \mathbb{C}^2 \left\vert\begin{matrix} \xi=g(1)\\\eta=g'(1)\end{matrix}\right.\right\}.
\end{align}
\end{subequations}
We also note that the norm $\norm{z(t)}_{X_{\text{sc}}}^2$ is exactly 
 twice the physical energy of the solution of $\Sigma_{\text{sc}}$ at time $t$, that is, $E_{\text{sc}}(t)=\frac{1}{2}\|z(t)\|_{X_{\text{sc}}}^2$.

\subsection{Combined force and torque control}
In this section we consider the case where the wind turbine tower model $\Sigma_{\text{sc}}$ has two control inputs: the force $u_1$, generated by
 an electrically driven mass within the nacelle, and the electric torque $u_2$, produced by 
the generator. 
In this
 setting the control
  operator $B_{\text{sc}}\in \mathcal L(\C^2,X_{\text{sc}})$ and the control input $v$ are given by
\begin{gather}
	B_{\text sc}=
		\begin{pmatrix}
			0&0&\frac{1}{m}&0\\
			0&0&0&\frac{1}{J}
		\end{pmatrix}^T, \qquad v(t)=\begin{pmatrix}
			u_1(t)\\u_2(t)
		\end{pmatrix}.
\end{gather}
The output signal $y$ in~\eqref{dessigma_output} consists of two components, one proportional to the velocity and the second proportional to the angular velocity of the nacelle. More precisely,
we let $a,b>0$ and define the observation operator $C_{\text{sc}}\in \mathcal L(X_{\text{sc}},\C^2)$ as
\begin{align*}
	C_{\text {sc}}=
	\begin{pmatrix}
		0&0&a&0\\
		0&0&0&b
	\end{pmatrix}.
\end{align*}

The following theorem shows
 that the static feedback $v = -y$ stabilizes the system $\Sigma_{\text{sc}}$ polynomially. 
\begin{theorem}\label{SCOLEdecay}
	Let $a$, $b>0$. Assume that there exist $\varepsilon>0$, $\delta>0$ and
$\zeta\in C^2[0,1]$
 such that 
for all $x\in [0,1]$ we have
	\begin{gather}\label{eq1}
\begin{gathered}
	\zeta(0)=0, \quad 2(1-\varepsilon)\rho(x)-(\rho\zeta)'(x)<-\delta<0,\\
		EI(x)[(1-\varepsilon)-2\zeta'(x)]+\frac{1}{2}(EI\zeta)'(x)<-\delta<0.
\end{gathered}	
\end{gather}%
 Then the wind turbine tower model $\Sigma_{\text{sc}}$ described by \eqref{dessigma} is stabilised polynomially by the static output feedback $v=-y$. More precisely,
	the energy of every solution of
	\begin{align}\label{SCOLE2}
		\begin{cases}
			\rho(x) w_{tt}(x,t)=-(EI(x)w_{xx}(x,t))_{xx}, &(x,t)\in (0,1)\times [0,\infty),\\
			w(0,t)=0, \quad w_x(0,t)=0,\\
			mw_{tt}(1,t)-(EIw_{xx})_x(1,t)=-aw_t(1,t),\\
			Jw_{xtt}(1,t)+EI(1)w_{xx}(1,t)=-bw_{xt}(1,t),\\
		\end{cases}
	\end{align}
	satisfies $E_{\textup{sc}}(t)\to 0$ as $t\to\infty$, and for classical solutions we have $E_{\textup{sc}}(t)=o(t^{-1})$ as $t\to\infty$.
\end{theorem}

\begin{rmk}
The role of condition \eqref{eq1} is to establish a resolvent estimate for $A_0$ in Theorem~\ref{thm1}.
We note that the condition \eqref{eq1} in particular holds when $\rho$ and $EI$ are constant. Indeed, in this case we can choose $\zeta(x)=2x$.
\end{rmk}

\begin{proof}[Proof of Theorem \ref{SCOLEdecay}]
Let $X_1:=H^2_l(0,1)\times L^2(0,1)$ with the norm defined by $\|(f,g)^T\|_{X_1}^2=\int_0^1EI(x)|f''(x)|^2dx+\int_0^1\rho(x)|g(x)|^2dx$, let $U=\C^2$ with the Euclidean norm, and 
let $X_2=\mathbb{C}^2$ with the norm 
    defined by $\|(z_1,z_2)^T\|_{X_2}^2=m|z_1|^2+J|z_2|^2$. We define the operators $L_1:D(L_1)\subset X_1\to X_1$ and $K_1,G_1\in\mathcal{L}(D(L_1),U)$ by
    \begin{align*}
      &L_1\begin{pmatrix}
        f\\g
      \end{pmatrix}=\begin{pmatrix}
        g\\
        -\frac{1}{\rho}(EIf'')''
      \end{pmatrix},\quad 
      &K_1\begin{pmatrix}
        f\\g
      \end{pmatrix}=\begin{pmatrix}
        -(EIf'')'(1)\\(EIf'')(1)
      \end{pmatrix},\quad G_1\begin{pmatrix}
        f\\g
      \end{pmatrix}=\begin{pmatrix}
        g(1)\\g'(1)
      \end{pmatrix},
    \end{align*}
    for all $(f,g)^T\in D(L_1)=\left(H^4(0,1)\cap H^2_l(0,1)\right)\times H^2_l(0,1)$. We define 
    \begin{align*}
      &A_2=\begin{pmatrix}
        -\frac{a}{m}&0\\
        0&-\frac{b}{J}
      \end{pmatrix}, \qquad
B_2=\begin{pmatrix}
        \frac{1}{m}&0\\
        0&\frac{1}{J}
      \end{pmatrix},
\qquad
 C_2=
\begin{pmatrix}
        1&0\\
        0&1
      \end{pmatrix},
    \end{align*}
  and $D_2=0\in \C^{2\times 2}$.
Then~\eqref{SCOLE2} can be recast as an abstract Cauchy problem of the form~\eqref{pro:cauchy} with state $z(t)=(w(\cdot,t),w_t(\cdot,t),w_t(1,t),w_{xt}(1,t))^T$ and with main operator $A$ of the form~\eqref{oper:A}. 
   The operator $G_1$ is surjective and, since $U$ is finite-dimensional,  $G_1$ is right-invertible.
    For all $(f,g)^T\in D(L_1)$ integration by parts yields 
    \begin{align*}
      \re  \left\langle L_1\begin{pmatrix}
        f\\g
      \end{pmatrix}, \begin{pmatrix}
        f\\g
      \end{pmatrix}\right\rangle_{X_1}&=\re  \Bigl(\langle EI g'',f''\rangle_{L^2}-\langle (EI f'')'',g\rangle_{L^2}\Bigr)\\
      &=\re \Bigl( -(EI f'')'(1)\overline{g(1)}+EI(1)  f''(1)\overline{g'(1)}\Bigr)\\
      &=\re\left\langle G_1\begin{pmatrix}
        f\\g
      \end{pmatrix}, K_1\begin{pmatrix}
        f\\g
      \end{pmatrix}\right\rangle_{U}.
    \end{align*}
In particular, the operator  $A_1 := {L_1}_{|\ker(G_1)}$ 
     is dissipative. Furthermore, the well-known Sobolev embedding theorems imply that the domain $D(A_1) =\ker(G_1)$
      is compactly embedded in $X_1$.  Consequently, $A_1$ has a compact resolvent and, in particular, it is maximally dissipative. Thus  
     $A_1$ generates a contraction semigroup on $X_1$
by the Lumer--Phillips theorem.  
These properties imply that $(G_1, L_1, K_1)$ is an impedance passive boundary node on $(U, X_1, U)$.  
     It is easy to see that 
$\re\langle A_2z+B_2u,z\rangle_{X_2} \leq \re\langle C_2z+D_2u,u\rangle_{U}$
for all $z\in X_2$, $u\in U$ 
  and 
therefore $(A_2,B_2,C_2, D_2)$
 is an impedance passive linear system on $(U,X_2,U)$. Moreover, with the condition \eqref{eq1}, it follows from \cite[Sec. 4.2]{cheDel87} that the restriction $A_0$ of $L_1$ to $\ker(G_1+K_1)$
  \begin{equation*}
    \begin{gathered}
      A_0=L_1,\\
   D(A_0)=\bigl\{(f,g)^T\in \left(H^4\cap H^2_l\right)\times H^2_l:  (EI f'')'(1)=g(1),\;(EI f'')(1)=-g'(1)\bigr\}
  \end{gathered}
  \end{equation*}
  generates an exponentially stable semigroup on $X_1$. 
Thus we have $i\mathbb{R}\subset \rho(A_0)$ and $\sup_{s\in\mathbb{R}}\|R(is,A_0)\|<\infty$. On the other hand, a direct computation shows that the transfer function $\mathbb{H}_2$ of $(A_2,B_2,C_2,D_2)$ satisfies
  \begin{align*}
    \re \mathbb{H}_2(is)=\frac{1}{\left(s^2-\frac{ab}{mJ}\right)^2+\left(\frac{b}{J}+\frac{a}{m}\right)^2s^2}\begin{pmatrix} \frac{a}{m}\left(s^2+\frac{b^2}{J^2}\right) &0\\ 0&\frac{b}{J}\left(s^2+\frac{a^2}{m^2}\right) \end{pmatrix}, \;s\in \R.
  \end{align*}
  Hence there exists $M>0$ such that $   \re \mathbb{H}_2(is)\geq \frac{M}{1+s^2}I_{\mathbb{C}^2}$
  for all $s\in\mathbb{R}$. Thus
 Theorem \ref{thm1} with $E=\R$ shows that the operator $A$ generates a contraction semigroup on $X_1\times X_2$ and that  $i\mathbb{R}\subset \rho(A)$ and $\|R(is,A)\|\lesssim 1+s^2$ for $s\in\mathbb{R}$.
  The claims therefore follow from the Arendt--Batty--Lyubich--Vu Theorem~\cite[Thm.~5.5.5(b)]{AreBat11book} and Theorem~\ref{thm:BorTom2010}.
  \end{proof} 

  \subsection{Torque control} 
In this section we consider the stabilization of the wind turbine tower model using only torque control, i.e., when $u_1\equiv 0$ in~\eqref{sigmasc}.
We use the boundary measurement $y(t)=bw_{xt}(1,t)$ with $b>0$ in the feedback.
In this case 
the control input is $v=u_2$ and
 the control operator and the observation operator in~\eqref{scolegenerator} are  
      $B_{\text sc}=(0,0,0,\frac{1}{J})^T$ and
$C_{\text sc}=( 0,0,0,b)$.
Our main result is presented below.
\begin{theorem}\label{thm:torque}
Assume that $b>0$.
  The wind turbine tower model $\Sigma_{\text{sc}}$ described by \eqref{dessigma} is stabilised polynomially by the static torque feedback $v=-y$ from the angular velocity of the nacelle. More precisely,
    the energy of every solution of the system \eqref{SCOLE2}, with $a=0$, satisfies $E_{\textup{sc}}(t)\to 0$ as $t\to\infty$, and for classical solutions we have $E_{\textup{sc}}(t)=o(t^{-1})$ as $t\to\infty$.
  \end{theorem}
  \begin{proof}
  We consider the space  $X_1:=H_l^2(0,1)\times L^2(0,1)\times \mathbb{C}$  with the norm defined by $ \|(f,g,h)^T\|_{X_1}^2=\int_0^1EI(x)|f''(x)|^2dx+\int_0^1\rho(x)|g(x)|^2dx+m|h|^2$.
 
Define the space $U=\C$ with the Euclidean norm and define $X_2=\C$ with the 
 norm $\|z\|_{X_2}^2=J|z|^2$. We define the operators $L_1:D(L_1)\subset X_1\to X_1$ and $K_1,G_1\in\mathcal{L}(D(L_1), U)$ by
  \begin{align*}
    &L_1\begin{pmatrix}
      f\\g\\h
    \end{pmatrix}=\begin{pmatrix}
      g\\
      -\frac{1}{\rho}(EIf'')''\\
      \frac{1}{m}(EIf'')'(1)
    \end{pmatrix},\quad 
    &K_1\begin{pmatrix}
      f\\g\\h
    \end{pmatrix}=(EIf'')(1),\quad G_1\begin{pmatrix}
      f\\g\\h
    \end{pmatrix}=g'(1),
  \end{align*}
  for all $(f,g,h)^T\in D(L_1)=\left\{(f,g,h)^T\in \left(H^4\cap H_l^2\right)\times H^2_l\times \mathbb{C}:\;h=g(1)\right\}$. 
Moreover, we define $A_2=-b/J\in \Lin(X_2)$, $B_2 = 1/J\in \Lin(U,X_2)$, $C_2=I\in \Lin(X_2,U)$, and $D_2=0\in \Lin(U)$.
Then~\eqref{SCOLE2} with $a=0$ can be formulated as an abstract Cauchy problem of the form~\eqref{pro:cauchy} with state $z(t)=(w(\cdot,t),w_t(\cdot,t),w_t(1,t),w_{xt}(1,t))^T$ and with the main operator $A$ of the form~\eqref{oper:A}.

To show that $(G_1,L_1,K_1)$ is an impedance passive boundary node, we note that
 $G_1$ is surjective, and since  $U$ is finite-dimensional, $G_1$ is right-invertible. Using integration by parts twice we get that for all $(f,g,h)^T\in D(L_1)$,
  \begin{align*}
    \re \left\langle L_1\begin{pmatrix}
      f\\g\\h
    \end{pmatrix}, \begin{pmatrix}
    f\\g\\h
  \end{pmatrix}\right\rangle_{X_1}&
=\re \Bigl(\langle EI g'',f''\rangle_{L^2}-\langle (EI f'')'',g\rangle_{L^2}+ (EIf'')'(1)\overline{h}\Bigr)\\
  & =\re \Bigl[(EIf'')(1)\overline{g'(1)}\Bigr]
=\re \left\langle G_1\begin{pmatrix}
    f\\g\\h
  \end{pmatrix}, K_1\begin{pmatrix}
  f\\g\\h
  \end{pmatrix}\right\rangle_{U}.
  \end{align*}
In particular
the restriction  $A_1$ of $L_1$ to $\ker G_1$ is dissipative. Furthermore, the Sobolev embedding theorem implies that the domain \( D(A_1) = \ker G_1 \)  is compactly embedded in \( X_1 \). 
  Consequently, \( A_1 \) has a compact resolvent and, in particular, 
is
 maximally dissipative. 
Thus \( A_1 \) generates a contraction semigroup on \( X_1 \) by the Lumer--Phillips theorem.
This completes the proof that  \( (G_1, L_1, K_1) \) is an impedance passive boundary node on $(U, X_1, U)$. Moreover, the restriction $A_0$ of $L_1$ to $\ker(G_1+K_1)$
  \begin{gather*}
    A_0=L_1,\\
    \mbox{ $\displaystyle D(A_0)
    =\left\{(f,g,h)^T\in\left(H^4\cap H_l^2\right)\times H^2_l\times \mathbb{C}: \; h=g(1),\;-(EIf'')(1)=g'(1)\right\}$}
  \end{gather*}
  generates an exponentially stable semigroup on $X_1$ by \cite[Thm. 4.2]{CheWan2007} and thus we have $i\mathbb{R}\subset \rho(A_0)$ and $\sup_{s\in \R}\|R(is,A_0)\|<\infty$. 
It is easy to check that
$(A_2,B_2,C_2,D_2)$
 is an impedance passive linear system on $(U,X_2,U)$ and its transfer function is given by $\mathbb{H}_2(is)=\frac{1}{is+b/J}$,
  and 
$\re \mathbb{H}_2(is)=\frac{b}{J}\bigl(\bigl(\frac{b}{J}\bigr)^2+s^2\bigr)\inv$.
Thus
 there exists $M>0$ such that 
$\re \mathbb{H}_2(is)\geq M(1+s^2)\inv$
 for all $s\in\mathbb{R}$. Hence, it follows from Theorem \ref{thm1} that the operator $A$ generates a contraction semigroup on $X_1\times X_2$ such that $i\mathbb{R}\subset \rho(A)$ and $\|R(is,A)\|\lesssim 1+s^2$ for $s\in\mathbb{R}$.
  The claims therefore follow from the Arendt--Batty--Lyubich--Vu Theorem~\cite[Thm.~5.5.5(b)]{AreBat11book} and Theorem \ref{thm:BorTom2010}.
  \end{proof}  

  \subsection{Force control} 
In this section we consider the stabilization using only force control, i.e., when $u_2\equiv 0$ in~\eqref{sigmasc}.
We use the boundary measurement $y(t)=aw_t(1,t)$ with $a>0$ in the feedback.
The control input is $v=u_1$ and
 the control operator  and the observation operator in~\eqref{scolegenerator} can be chosen as 
  \begin{align*}
    B_{\text {sc}}=\begin{pmatrix}
      0&0&\frac{1}{m}&0
    \end{pmatrix}^T,  \qquad  C_{\text {sc}}=\begin{pmatrix}
      0&0&a&0
    \end{pmatrix}.
  \end{align*}
  In this case we assume that \( \Sigma_{\text{sc}} \) is uniform, i.e., \( EI \) and \( \rho \) in \eqref{sigmasc} are strictly positive constants. 
The proof of our main result below
 is based on Theorem \ref{thm1}.
The role of condition \eqref{cond} is to guarantee that the resolvent of $A_0$ in Theorem~\ref{thm1} exists and is uniformly bounded on $i\R$.
We note that
if the condition~\eqref{cond} is not satisfied, then $A_0$ has at least one spectral point on $i\R$, see \cite[Rem.~4]{che2003}.

  \begin{theorem}\label{thm:force}
Assume that $a>0$, that $EI>0$ and $\rho>0$  in \eqref{sigmasc} are constant and that
    \begin{align}\label{cond}
      J\frac{EI}{\rho}\neq\frac{\sinh(k\pi)}{(k\pi)^3\left(\cosh(k\pi)-(-1)^k\right)}, \quad k\in \mathbb{N}.
    \end{align}
    Then the wind turbine tower model $\Sigma_{\text{sc}}$ described by \eqref{dessigma} is  stabilised polynomially by the static force feedback $v=-y$ from the velocity of the nacelle. More precisely,
    the energy of every solution of the system \eqref{SCOLE2}, with $b=0$, satisfies $E_{\textup{sc}}(t)\to 0$ as $t\to\infty$, and for classical solutions we have $E_{\textup{sc}}(t)=o(t^{-1})$ as $t\to\infty$.
  \end{theorem}
  \begin{proof}
  We consider the space  $X_1:=H_l^2(0,1)\times L^2(0,1)\times \mathbb{C}$ equipped with the norm
 $ \|(f,g,h)^T\|_{X_1}^2=EI\int_0^1|f''(x)|^2dx+\rho\int_0^1|g(x)|^2dx+J|h|^2.$
    We define $U=\mathbb{C}$ with the Euclidean norm and $X_2=\C$ with the 
 norm $\|z\|_{X_2}^2=m|z|^2$.	 We define  $L_1:D(L_1)\subset X_1\to X_1$ and $K_1,G_1\in\mathcal{L}(D(L_1),U)$ by
    \begin{align*}
      &L_1\begin{pmatrix}
        f\\g\\h
      \end{pmatrix}=\begin{pmatrix}
        g\\
        -\frac{EI}{\rho}f''''\\
        -\frac{EI}{J}f''(1)
      \end{pmatrix},\;
      &K_1\begin{pmatrix}
        f\\g\\h
      \end{pmatrix}=-EIf'''(1),\quad G_1\begin{pmatrix}
        f\\g\\h
      \end{pmatrix}=g(1),
    \end{align*}
    \sloppy for all $(f,g,h)^T\in D(L_1)=\left\{(f,g,h)^T\in\left(H^4\cap H_l^2\right)\times H^2_l\times \mathbb{C}:\; h=g'(1)\right\}$. 
Moreover, $A_2=-a/m\in \Lin(X_2)$, $B_2=1/m\in \Lin(U,X_2)$, $C_2=I\in \Lin(X_2,U)$ and $D_2=0\in \Lin(U)$.
Then~\eqref{SCOLE2} with $b=0$ can be formulated as an abstract Cauchy problem of the form~\eqref{pro:cauchy} with state $z(t)=(w(\cdot,t),w_t(\cdot,t),w_{xt}(1,t),w_{t}(1,t))^T$ and with the main operator $A$ of the form~\eqref{oper:A}. 
    It is clear that the operator $G_1$ is surjective and since $U$ is finite-dimensional then $G_1$ is right invertible.  Using integration by parts twice, we get that for all $(f,g,h)^T\in D(L_1)$ 
  \begin{align*}
    \re \left\langle L_1\begin{pmatrix}
      f\\g\\h
    \end{pmatrix}, \begin{pmatrix}
    f\\g\\h
  \end{pmatrix}\right\rangle_{X_1}
&= \re \Bigl(EI\langle g'',f''\rangle_{L^2}-EI\langle  f'''',g\rangle_{L^2}
 -EI f''(1)\overline{h}\Bigr)\\
  & =\re \Bigl[-EI f'''(1)\overline{g(1)}\Bigr]
=\re \left\langle G_1\begin{pmatrix}
    f\\g\\h
  \end{pmatrix}, K_1
\begin{pmatrix}
  f\\g\\h
  \end{pmatrix}
\right\rangle_{U}.
  \end{align*}
In particular
 the restriction $A_1$ of $L_1$ to $\ker G_1$ is dissipative. Furthermore, the Sobolev embedding theorem implies that the domain $ D(A_1) = \ker G_1 $  is compactly embedded in $ X_1 $. 
  Consequently,  $A_1$ has a compact resolvent and, in particular, is maximally dissipative. Thus
   $A_1$ generates a contraction semigroup on $X_1$ by the Lumer--Phillips theorem. Hence, the system $(G_1, L_1, K_1)$ is an impedance passive boundary node on $(U, X_1, U)$. On the other hand, the restriction $A_0$ of  $L_1$ to $\ker (G_1+K_1)$ 
is given by
    \begin{gather*}
      \begin{gathered}
      A_0=L_1\\
       D(A_0)=\left\{(f,g,h)^T\in\left(H^4\cap H_l^2\right)\times H^2_l\times \mathbb{C}: \; h=g'(1), ~~EIf'''(1)=g(1)\right\}.
      \end{gathered}
    \end{gather*}
By \cite[Thm. 4]{che2003} this operator generates an exponentially stable semigroup on $X_1$ if (and only if) the condition \eqref{cond} holds. Thus $i\mathbb{R}\subset\rho(A_0)$ and $\sup_{s\in \R}\|R(is,A_0)\|<\infty$.
    Moreover,
$(A_2,B_2,C_2,D_2)$
 is an impedance passive linear system on $(U,X_2,U)$. Its transfer function 
     is $\mathbb{H}_2(is)=\frac{1}{is+a/m}$ and thus
$\re \mathbb{H}_2(is)=\frac{a}{m}\bigl(\bigl(\frac{a}{m}\bigr)^2+s^2\bigr)\inv$ for $s\in\R$.
     Hence there exists $M>0$ such that 
     $\re \mathbb{H}_2(is)\geq M(1+s^2)\inv$ for all $s\in\mathbb R$.  Theorem \ref{thm1} with $E=\R$ implies that the operator $A$ generates a contraction semigroup on $X_1\times X_2$ and that $i\mathbb{R}\subset \rho(A)$ and $\|R(is,A)\|\lesssim 1+s^2$ for  $s\in \mathbb{R}$.
  The claims therefore follow from the Arendt--Batty--Lyubich--Vu Theorem~\cite[Thm.~5.5.5(b)]{AreBat11book} and Theorem \ref{thm:BorTom2010}.
    \end{proof}

    \subsection{Tuned mass damper (TMD)}
    In this section we study the stabilization of the monopile wind turbine $\Sigma_{\text{sc}}$ using a tuned mass damper (TMD) system. The TMD consists 
    of a large mass placed either on the nacelle floor using wheels or racks or suspended above the floor using cables. The mass and spring are carefully tuned
     to a specific frequency, allowing the TMD mass to oscillate at that frequency. 
The damper then reduces vibrations by dissipating the system's energy as heat.
  
  The mathematical model $\Sigma_\text{TMD}$ for the monopile wind turbine tower $\Sigma_{\text{sc}}$ stabilized by a TMD in the fore–aft plane
has the form~\cite{TonZha2017,ZhaWei2017}
  \begin{align}\label{SCOLE3}
    \begin{cases}
      \rho(x) w_{tt}(x,t)=-(EI(x)w_{xx})_{xx}(x,t), \qquad(x,t)\in (0,1)\times [0,\infty),\\
      w(0,t)=0, \quad w_x(0,t)=0,\\
      mw_{tt}(1,t)-(EIw_{xx})_x(1,t)=k_1[p(t)-w(1,t)]+d_1[p_t(t)-w_t(1,t)],\\
      Jw_{xtt}(1,t)+EI(1)w_{xx}(1,t)=0,\\
      m_1p_{tt}(t)=k_1[w(1,t)-p(t)]+d_1[w_t(1,t)-p_t(t)],
    \end{cases}
  \end{align}
  where $m_1>0$, $k_1>0$ and $d_1>0$ are the mass, spring constant and damping coefficient of the TMD.
   Here $p(t)-w(1,t)$ describes the displacement of the mass component of the TMD with respect to the nacelle and $p_t(t)$ 
   describes the mass component's translational velocity with respect to the earth. The monopile wind turbine tower and the 
   TMD subsystems are interconnected through the translational velocity of the nacelle $w_t(1,t)$ and the 
  force $k_1(p(t)-w(1,t))+d_1(p_t(t)-w_t(1,t))$ generated by the spring and the damper in the TMD system.
  The system $\Sigma_{\text{TMD}}$ can be formulated as a Cauchy problem
 with the state
  \begin{align*}
      z(t)=\Bigl(w(\cdot,t), w_t(\cdot,t), w_t(1,t), w_{xt}(1,t), p(t)-w(1,t), p_t(t)\Bigr)^T.
  \end{align*}
  The natural energy state space of the solutions is 
    $\XTMD=H^2_l(0,1)\times L^2(0,1)\times\mathbb{C}^4$ equipped with the norm
  \begin{align*}
    \|f\|_{\XTMD}^2=\int_0^1\left[EI(x)|f_1''(x)|^2+\rho(x)|f_2(x)|^2\right]dx+m|f_3|^2+J|f_4|^2+k_1|f_5|^2+m_1|f_6|^2
  \end{align*}
  for $f=(f_1,f_2,f_3,f_4,f_5,f_6)^T\in \XTMD$. 
We note that the physical energy of the solution of~\eqref{SCOLE3} satisfies 
$E_{\textup{TMD}}(t)=\frac{1}{2}\|z(t)\|_{\XTMD}^2$.
The operator $A:D(A)\subset \XTMD\to \XTMD$ in the abstract Cauchy problem~\eqref{pro:cauchy} is defined by
\begin{subequations}
\label{A:TMD}
  \begin{align}
 &A\begin{pmatrix}
      f_1\\f_2\\f_3\\f_4\\f_5\\f_6
    \end{pmatrix}=
    \begin{pmatrix}
      f_2\\
      -\frac{1}{\rho}(EIf_1'')''\\
      \frac{1}{m}(EIf_1'')'(1)-\frac{d_1}{m}f_3+\frac{k_1}{m}f_5+\frac{d_1}{m}f_6\\
      -\frac{1}{J}(EIf_1'')(1)\\
      -f_3+f_6\\
      \frac{d_1}{m_1}f_3-\frac{k_1}{m_1}f_5-\frac{d_1}{m_1}f_6
    \end{pmatrix}\\[1ex]
 &   D(A)=\{(f_1,\ldots,f_6)^T\in \left(H^2_l\cap H^4\right)\times H^2_l\times \mathbb{C}^4: \;
f_4=f_2'(1),~f_3=f_2(1)\}.%
  \end{align}%
\end{subequations}%

Our main result below shows that if the parameters $EI$ and $\rho$ are constant, then the TMD stabilizes the wind turbine tower model polynomially.

  \begin{theorem}\label{thm:TMD}
    Assume that $EI$ and $\rho$ in \eqref{SCOLE3} are strictly positive constants and sastisfies
    \begin{align}
      J\frac{EI}{\rho}\neq\frac{\sinh(k\pi)}{(k\pi)^3\left(\cosh(k\pi)-(-1)^k\right)}, \quad k\in \mathbb{N}.
    \end{align}
    Then the energy of every solution of the monopile wind turbine with the TMD described by \eqref{SCOLE3} satisfies $E_{\textup{TMD}}(t)\to 0$ as $t\to\infty$, and for classical solutions we have $E_{\textup{TMD}}(t)=o(t^{-1})$ as $t\to\infty.$
  \end{theorem}

  \begin{proof}
    Let $X_1=X_\text{sc} = H^2_l(0,1)\times L^2(0,1)\times \mathbb{C}^2 $ with the inner product defined by~\eqref{eq:Xscnorm},
let $U=\mathbb{C}$ with the Euclidean norm
and define
 $A_1:D(A_1)\subset X_1\to X_1$  so that  $A_1=A_\text{sc}$
 in \eqref{scolegenerator}. Then the operator $A$ in \eqref{A:TMD} has the structure
    \begin{align*}
      A=\begin{pmatrix}
        A_1-B_1D_2C_1&B_1C_2\\
        -B_2C_1&A_2
      \end{pmatrix},
    \end{align*}
    where 
      $B_1=(0,0,\frac{1}{m},0)^T\in \Lin(U,X_1)$,  $C_1=(0,0,1,0)\in \Lin(X_1,U)$,
    \begin{gather*}
      A_2=\begin{pmatrix}
        0 &1\\
        -\frac{k_1}{m_1}&-\frac{d_1}{m_1}
      \end{pmatrix}, \qquad B_2=\begin{pmatrix}
        1\\-\frac{d_1}{m_1}
        \end{pmatrix}, \qquad C_2=\begin{pmatrix}
          k_1& d_1
      \end{pmatrix}, \qquad D_2=d_1.
    \end{gather*}
Since the $A_1$ is skew-adjoint with respect to the inner product on $X_1$ (see \cite[Prop. 1.1]{GuoIva2005}) and since $C_1=B_1^*\in \Lin(X_1,U)$, it is easy to check that $(A_1, B_1, C_1, D_1)$ with $D_1 = 0$ is an impedance passive linear system on $(U, X_1, U)$.
Define the space $X_2=\mathbb{C}^2$ with the norm $\|(z_1,z_2)^T\|_{X_2}^2 = k_1|z_1|^2 + m_1|z_2|^2$. Then we have
  \begin{align*}
  &\re \langle A_2z+B_2u,z\rangle_{X_2}=-d_1|z_2|^2+k_1\re \bigl(u\overline{z_1}\bigr)-d_1\re \bigl(u\overline{z_2}\bigr),\\	
  & \re \langle C_2z+D_2u,u\rangle_{U}=k_1\re \bigl(z_1\overline{u}\bigr)+d_1\re \bigl(z_2\overline{u}\bigr)+d_1|u|^2,
  \end{align*}
  for $z=(z_1,z_2)^T\in X_2$ and $u\in U$. We have $\re \langle A_2z+B_2u,z\rangle_{X_2}-\re \langle C_2z+D_2u,u\rangle_{U}=-d_1|z_2+u|^2\leq0$ 
and $(A_2,B_2,C_2,D_2)$ is an impedance passive linear system on $(U,X_2,U)$.
  Moreover, since $\det(\lambda - A_2) = \lambda^2 + \frac{d_1}{m_1} \lambda + \frac{k_1}{m_1}$ for $\lambda \in \C$, the Routh--Hurwitz implies that $\sigma(A_2) \subset \mathbb{C}_-$.
A direct computation shows that the transfer function of
   $(A_2,B_2,C_2,D_2)$ satisfies
\begin{align*}
  \re \mathbb{H}_2(is)=\frac{d_1m_1^2s^4}{m(d_1^2s^2+(k_1-m_1s^2)^2)},
\end{align*}
for all $s\in\mathbb{R}$.
For any given $\varepsilon>0$ there exists a constant $c_\varepsilon>0$ such that 
    $\re \mathbb{H}_2(is)\geq c_\varepsilon$ 
  for all $s\in\mathbb{R}\setminus(-\varepsilon, \varepsilon)$. 
Theorem \ref{thm:force} shows that if $K=I$, then
 the operator $A_K=A_1-B_1KC_1$  with domain $D(A_K)=D(A_1)$ 
 generates a contraction semigroup
$\mathbb T_K$ on $X_1$ and that $\norm{\mathbb T_Kz_0}^2=o(t\inv)$ as $t\to\infty$ for all $z_0\in D(A_1)$.
Theorem \ref{thm:BorTom2010} thus implies that
  $i\mathbb{R}\subset \rho(A_K)$ and $\|R(is,A_K)\|\lesssim 1+s^2$ for  $s\in\mathbb{R}$. It follows from Theorem \ref{thm:AbsCouplingResEst} with $E=\{s\in\mathbb R: \; |s|>\varepsilon\}$ that $A$ generates a contraction semigroup on $\XTMD$ and that $\left\{is\in \R: \; |s|>\varepsilon\right\}\subset \rho(A)$ and $\|R(is,A)\|\leq M_\varepsilon (1+s^2)$ for some $M_{\varepsilon}>0$ and for $|s|>\varepsilon$. 
We will now show that $0\in\rho(A)$. 
The Sobolev embedding theorem implies that $A$ has compact resolvents, and therefore it is sufficient to show that $\ker A=\{0\}$.
  Let $f=(f_1,f_2,f_3,f_4,f_5,f_6)^T\in D(A)$ such that ${A}f=0$. Then
  \begin{align}
        &f_2=0, \quad f_1''''=0,
\quad f_3=f_2(1), \quad f_4=f_2'(1)
\label{1}\\
      &-d_1f_3+k_1f_5+d_1f_6+EIf_1'''(1)=0\label{3}\\
      &f_1''(1)=0, \quad f_1(0)=f_1'(0)=f_2(0)=f_2'(0)=0\label{4}\\
      &-f_3+f_6=0, \quad d_1f_3-k_1f_5-d_1f_6=0\label{5}.
  \end{align}	
  Using \eqref{1}
 we get $f_4=f_3=0$ and then \eqref{5} implies that
 $f_6=f_5=0$ as well.
Moreover, \eqref{1} implies that $f_1$ is a polynomial of degree $3$, and using \eqref{4}
and the fact that \eqref{3} implies $f_1'''(1)=0$
 we get that $f_1=0$. Thus $0\in\rho(A)$.
  Choosing $\epsilon>0$ above sufficiently small we deduce that $i\mathbb{R}\subset \rho(A)$ and $\|R(is,A)\|\lesssim {1+s^2}$ for $s\in\mathbb{R}$.
  The claims therefore follow from the Arendt--Batty--Lyubich--Vu Theorem~\cite[Thm.~5.5.5(b)]{AreBat11book} and Theorem \ref{thm:BorTom2010}.
  \end{proof}

  \section{Polynomial stability of the wind turbine tower model in the side-side plane: Hydrostatic transmission}\label{Sec:side-side}
  To investigate the effect of the vibration torque transferred from the turbine to the nacelle by the hydraulic pump, we
 study a model~\cite{RajHsu2014, FarEla2018} of the wind turbine tower in the plane of the turbine blades. 
The model $\Sigma_{\text{hd}}$ consists of the nonuniform SCOLE model system coupled with a fluid power transmission model (see Figure \ref{figure02}) and has the form 
   \begin{figure}
    \centering
    \includegraphics[scale=0.5]{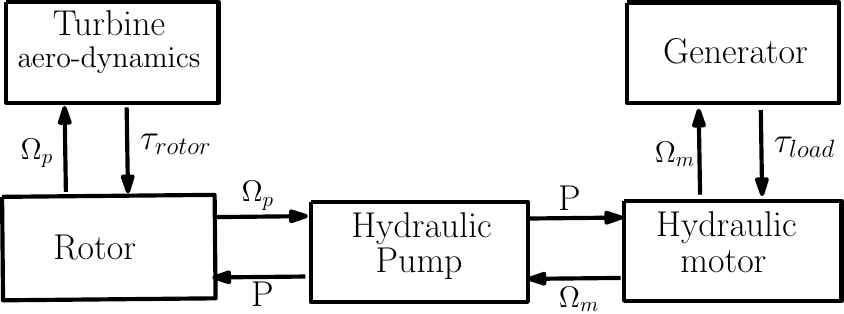}
    \caption{Flow diagram of the wind turbine with hydrostatic transmission.}\label{figure02}
  \end{figure}
has the form
\begin{subequations}
\label{eq:SCOLEHydro}
  \begin{numcases}{}
    \rho(x) w_{tt}(x,t)=-(EI(x)w_{xx})_{xx}(x,t), \quad(x,t)\in (0,1)\times [0,\infty),\label{0001}\\
      w(0,t)=0, \quad w_x(0,t)=0,\label{0002}\\
      mw_{tt}(1,t)-(EIw_{xx})_x(1,t)=0,\label{0003}\\
      Jw_{xtt}(1,t)+EI(1)w_{xx}(1,t)=D_pP(t)+B_p\Omega_p(t)+D_mP(t)\label{0004}\\
      \qquad\qquad\qquad \qquad\qquad \qquad\qquad-B_m \Omega_m(t)-\tau_{\text{load}}(t),\notag\\
      \dot \Omega_p(t)=\frac{1}{J_T}(\tau_{\text{rotor}}(t)-B_p\Omega_p(t)-D_pP(t))-w_{xtt}(1,t),\label{0005}\\
      \dot \Omega_m(t)=\frac{1}{J_G}(D_mP(t)-B_m\Omega_m(t)-\tau_{\text{load}}(t))+w_{xtt}(1,t),\label{0006}\\
      Q_p(t)=D_p\Omega_p(t)-k_{\text{leak},p}P(t),\label{0007}\\
      Q_m(t)=D_m\Omega_m(t)+k_{\text{leak},m}P(t),\label{0008}\\
      \dot P(t)=\frac{\beta}{V}(Q_p(t)-Q_m(t)).\label{0009}
  \end{numcases}
\end{subequations}
  The equations \eqref{0005}--\eqref{0009} describe the hydrostatic transmission system. More precisely, equation \eqref{0007} is the flow equation for the pump, where $Q_p$
   is the pump flow delivery, $D_p>0$ is the pump displacement, $k_{\text{leak,p}}\geq 0$ is the pump leakage coefficient, $P$ is the differential pressure across the pump,
    and $\Omega_p$ is the pump's angular velocity. Equation \eqref{0005} represents the torque balance between the wind turbine rotor and the pump, where the angular 
    velocity of the nacelle, $\omega_{xt}(1,t)$, is also taken into consideration. In this equation, the term $J_T \Omega_p$ represents the pump's inertial torque, 
    where $J_T>0$ is the total inertia of the wind turbine rotor and the pump. The term $B_p \Omega_p$ represents the friction torque, where $B_p$ is the damping 
    coefficient of the pump, and $D_p P$ is the pump reaction torque.
  
  Equation \eqref{0008} describes the hydraulic flow supplied to the hydraulic motor, where $Q_m$ is the motor flow delivery, $D_m>0$ is the motor displacement,
   $k_{\text{leak,m}}\geq0$ is the motor leakage coefficient, $P$ is the differential pressure across the motor, and $\Omega_m$ is the motor's angular velocity.
    Equation \eqref{0006} provides the balance between the driving torque and the braking torque for the motor shaft. The term $J_G \Omega_m$ is the motor's
     inertial torque, where $J_G$ is the total inertia of the motor rotor and the generator. The term $B_m \Omega_m$ represents the motor's friction torque,
      where $B_m$ is the damping coefficient of the motor, $\tau_{\text{load}}$ is the generator load torque, and $D_m P$ is the rotor torque.
  Finally, equation \eqref{0009} represents the hydraulic circuit equation, where $\frac{V}{\beta}$ is the fluid capacitance of the hydraulic circuits,
   $V$ is the total compressed volume from the hydraulic pump to the hydraulic motor, and $\beta$ is the effective bulk modulus of the hydraulic fluid.
  
In our analysis it is necessary to assume that at least one of the damping coefficients, either that of the pump, $B_p$, or that of the motor rotor, $B_m$, is nonzero. 
We will now reformulate the system $\Sigma_{\text{hd}}$ described by \eqref{eq:SCOLEHydro} as abstract Cauchy problem of the form~\eqref{pro:cauchy} with the state
\begin{align*}
 z(t)=     \bigl(w(\cdot,t),w_t(\cdot,t),w_t(1,t),w_{xt}(1,t),\Omega_p(t)+w_{xt}(1,t),\Omega_m(t)-w_{xt}(1,t),P(t)\bigr)^T.
\end{align*}
  The energy state space of the solutions of $\Sigma_{\text{hd}}$ is 
    $X_{\text{hd}}=H^2_l(0,1)\times L^2(0,1)\times \mathbb{C}^5$
equipped
  with the norm
  \begin{align*}
    \|f\|_{X_{\text {hd}}}^2&=\int_0^1\left[EI(x)|f_1''(x)|^2+\rho(x)|f_2(x)|^2\right]dx+m|f_3|^2+J|f_4|^2+J_T|f_5|^2\\
    &\quad\qquad+J_G|f_6|^2+\frac{\beta}{V}|f_7|^2
  \end{align*}
for $f= (f_1,f_2,f_3,f_4,f_5,f_6,f_7)^T\in \Xhd$.
  We note that the physical energy of the solution of  $\Sigma_{\text{hd}}$ satisfies $E_{\text{hd}}(t)=\frac{1}{2}\|z(t)\|_{X_{\text {hd}}}^2$.
  Denoting $k_{\text{leak}}:=k_{\text{leak},p}+k_{\text{leak},m}$, the main operator $A_{\text{hd}}$ of the abstract Cauchy problem is defined by
  \begin{align*}
     &A_{\text{hd}}f=\begin{pmatrix}
      f_2\\
      \frac{1}{\rho(x)}\left(EI(x)f''_1\right)''\\
      \frac{1}{m}\left(EIf''_1\right)'(1)\\
      -\frac{EI(1)}{J}f''_1(1)+\frac{B_p}{J}f_5-\frac{B_p+B_m}{J}f_4-\frac{B_m}{J}f_6+\frac{D_p+D_m}{J}f_7\\
      -\frac{B_p}{J_T}f_5+\frac{B_p}{J_T}f_4-\frac{D_p}{J_T}f_7\\
      \frac{D_m}{J_G}f_7-\frac{B_m}{J_G}f_6-\frac{B_m}{J_G}f_4\\
      \frac{\beta}{V}(D_pf_5-(D_p+D_m)f_4-D_mf_6-k_{\text{leak}}f_7)
       \end{pmatrix}\\[1ex]
    &D(A_{\text{hd}})=\bigl\{(f_1,\ldots,f_7)^T\in \left(H^2_l\cap H^4\right)\times H^2_l\times \mathbb C^5:\; f_2(1)=f_3, \;f_2'(1)=f_4\bigr\}.
  \end{align*}

Our main result below shows that in the case where the parameters of~\eqref{eq:SCOLEHydro} satisfy $J=J_G=J_T=1$ the system $\Sigma_{\text{hd}}$ 
without control input ($\tau_{\text{load}}(t)\equiv 0$) and without external disturbances ($\tau_{\text{rotor}}(t)\equiv 0$) 
is polynomially stable. 

  \begin{theorem}\label{thm43}
    Assume that $J=J_G=J_T=1$ and $B_p+B_m>0$. 
The hydraulic wind turbine model $\Sigma_{\text{hd}}$ described by \eqref{eq:SCOLEHydro}
 is polynomially stable. 
 More precisely, the energy of every solution corresponding to $\tau_{\text{rotor}}(t)\equiv 0$ and $\tau_{\text{load}}(t)\equiv 0$ satisfies $E_{\textup{hd}}(t)\to 0$ as $t\to 0$, and for classical solutions we have $E_{\textup{hd}}(t)=o(t^{-1})$ as $t\to\infty$.
  \end{theorem}

The proof of Theorem~\ref{thm43} is based on Theorem \ref{thm:AbsCouplingResEst} and the observation that if $J=J_G=J_T=1$, then the operator $A_{\text{hd}}$ has the structure

  \begin{align}
\label{eq:Ahd_def}
    A_{\text{hd}}=\begin{pmatrix}
      A_1-B_1D_2C_1&B_1C_2\\
      -B_2C_1&A_2
    \end{pmatrix},
  \end{align}
  where $A_1=A_{\text{sc}}$, with $A_{\text{sc}}$ in \eqref{scolegenerator}, and  $B_1=
    (0,0,0,1)^T$, $C_1= (0,0,0,1)$,
  \begin{align*}
  &A_2=	\begin{pmatrix}
      -B_p&0&-D_p\\
      0&-B_m&{D_m}\\
      \frac{\beta}{V}D_p&-\frac{\beta}{V}D_m&-\frac{\beta}{V}k_\text{leak}
    \end{pmatrix},\quad
  	B_2=\begin{pmatrix}
      {-B_p}&B_m& \frac{\beta}{V}(D_p+D_m)
    \end{pmatrix}^T,\\
    &C_2=\begin{pmatrix}
    {B_p}& -B_m& (D_p+D_m)
    \end{pmatrix},\quad
  	D_2=B_m+B_p.
  \end{align*}
  We define $X_2=\mathbb{C}^3$ equipped with the norm  $\|(z_1,z_2,z_3)^T\|_{X_2}^2=|z_1|^2+|z_2|^2+\frac{V}{\beta}|z_3|^2$. The following lemma describes the properties of  $(A_1,B_1,C_1, D_1)$ and $(A_2,B_2,C_2,D_2)$.
  \begin{lemma} \label{lem:stable}
    Assume that $J=J_G=J_T=1$ and $B_p+B_m>0$. 
\begin{enumerate}
      \item[\textup{(a)}] 
The tuple $(A_1,B_1,C_1, D_1)$ with $D_1=0$ is an impedance passive linear system on $(\mathbb{C},X_1,\mathbb{C})$.
      \item[\textup{(b)}] The tuple $(A_2,B_2,C_2,D_2)$ is an impedance passive linear system  on $(\mathbb{C},X_2,\mathbb{C})$ and $\sigma(A_2)\subset \mathbb{C}_-$.
    \end{enumerate}
  \end{lemma}
  \begin{proof}
  Part (a) follows from the fact that $A_1$ is skew-adjoint (see \cite[Prop. 1.1]{GuoIva2005}) and that $C_1=B_1^*$ and $D_1=0$. To prove part (b), let $z=(z_1,z_2,z_3)^T\in X_2$ and $u\in \mathbb{C}$. We have
\begin{align*}
\langle A_2z+B_2u,z\rangle_{X_2}&=-B_p|z_1|^2-B_m|z_2|^2-k_\text{leak}|z_3|^2-B_pu\overline{z_1}\\
&\quad+B_mu\overline{z_2}+(D_p+D_m)u\overline{z_3},\\
\langle C_2z+D_2u,u\rangle_{\mathbb C} 
&=B_pz_1\overline{u}-B_mz_2\overline{u}+(D_p+D_m)z_3\overline{u}+(B_p+B_m)|u|^2
  \end{align*}
and thus
      \begin{align*}
        \MoveEqLeft[5]\re\langle A_2z+B_2u,z\rangle_{X_2}- \re\langle C_2z+D_2u,u\rangle_{\mathbb C}\\
&= -B_p|z_1+u|^2-B_m|z_2-u|^2-k_\text{leak}|z_3|^2\leq0.
      \end{align*}
Since $X_2$ is finite-dimensional, this implies that $(A_2,B_2,C_2,D_2)$ is an impedance passive linear system on $(\mathbb{C},X_2,\mathbb{C})$. Moreover, for $\lambda\in\mathbb{C}$ we have
   \begin{align*}
  \det(\lambda - A_2)&= 
  \lambda^3 + (D_m^2 + D_p^2 + B_m k_{\text{leak}} + B_p k_{\text{leak}} + B_m B_p) \lambda \\
  &\quad + (B_m + B_p + k_{\text{leak}}) \lambda^2  + B_p D_m^2 + B_m D_p^2 + B_m B_p k_{\text{leak}}.
  \end{align*}
Since $B_p+B_m+k_{\text{leak}}>0$ and $D_p,D_m>0$ we have
 \begin{align*}
    B_p D_m^2 + B_m D_p^2 + B_m B_p k_{\text{leak}}&\leq (B_p  + B_m  +  k_{\text{leak}}) \max\{D_m^2,D_p^2, B_m B_p\}\\
  &\hspace{-1.3cm}< (B_p  + B_m  +  k_{\text{leak}}) (D_m^2+D_p^2+ B_m B_p)\\
  &\hspace{-1.3cm}\leq (B_p  + B_m  +  k_{\text{leak}}) (D_m^2 + D_p^2 + B_m k_{\text{leak}} + B_p k_{\text{leak}} + B_m B_p). 
  \end{align*}
Since $B_p+B_m>0$, $D_p,D_m>0$, and $B_p, k_{\text{leak}}\geq0$, the above estimate and the Routh--Hurwitz criterion imply that $\det(\lambda-A_2)\neq 0$ for all $\gl\in\overline{\C_+}$, and thus
 $\sigma(A_2)\subset \C_-$.
  \end{proof} 
  \begin{proof}[Proof of Theorem~\ref{thm43}]
We have from Lemma~\ref{lem:stable}
that $(A_1,B_1,C_1, D_1)$ with $D_1=0$ and $(A_2,B_2,C_2,D_2)$ are impedance passive linear systems and $\sigma(A_2)\subset \mathbb{C}_-$.
To apply Theorem \ref{thm:AbsCouplingResEst} we will analyse the transfer function $\mathbb H_2$ of $(A_2,B_2,C_2,D_2)$.
    A direct computation shows that
 $\re \mathbb{H}_2(is)= \frac{s^2 n(s)}{d(s)}$ for all $s\in \R\setminus \{0\}$, where
  \begin{align*}
    \MoveEqLeft[1] n(s)=(B_m+B_p)s^4\\
    &+\bigl[(B_mB_p+k_{\text{leak}}^2)(B_m+B_p)+(D_m+D_p)^2k+2(B_mD_p-B_pD_m)(D_m-D_p)\bigr]s^2\\
    &+\left[(B_mD_p^2+B_pD_m^2)(D_m-D_p)^2+B_mB_p(B_m+B_p)k_{\text{leak}}^2\right.\\
    &\quad\left.+\left(2B_mB_p(D_m^2+D_p^2)+(B_mD_p-B_pD_m)^2\right)k_{\text{leak}}\right]
  \end{align*}
  \begin{align*}
    d(s)&= s^6+ \left[B_m^2 + B_p^2 - 2 D_m^2 - 2 D_p^2  + k_{\text{leak}}^2\right] s^4 \\
    &\quad+ \left[B_m^2 B_p^2 - 2 B_m^2 D_p^2 + B_m^2 k_{\text{leak}}^2 + D_p^4 - 2 B_p^2 D_m^2 + B_p^2 k_{\text{leak}}^2 + 2 B_p D_p^2 k_{\text{leak}} \right.\\
    &\quad\left.+ D_m^4 + 2 D_m^2 D_p^2 + 2 B_m D_m^2 k_{\text{leak}} \right]s^2+B_m^2 B_p^2 k_{\text{leak}}^2 + 2 B_m^2 B_p D_p^2 k_{\text{leak}} \\
    & \quad+ B_m^2 D_p^4+ 2 B_m B_p^2 D_m^2 k_{\text{leak}} + 2 B_m B_p D_m^2 D_p^2 + B_p^2 D_m^4.
  \end{align*}
  We want to show that $n(s)\neq 0$ for all $s\neq 0$.
We have 
$n(s)=a_2 s^4 + a_1 s^2 + a_0$, where
$a_2=B_m+B_p>0$ and
  \begin{align*}
    &a_1=(B_mB_p+k_{\text{leak}}^2)(B_m+B_p)+(D_m+D_p)^2k_{\text{leak}}+2(B_mD_p-B_pD_m)(D_m-D_p)\\
    &a_0=(B_mD_p^2+B_pD_m^2)(D_m-D_p)^2+B_mB_p(B_m+B_p)k_{\text{leak}}^2\\
    &\qquad\quad+\left(2B_mB_p(D_m^2+D_p^2)+(B_mD_p-B_pD_m)^2\right)k_{\text{leak}}.
  \end{align*}
We first consider the case where $a_0=0$.
 Since $B_p+B_m>0$ and $D_p,D_m>0$, this implies that $D_m=D_p$ 
and further that $a_1 = (B_mB_p+k_{\text{leak}}^2)(B_m+B_p)+(D_m+D_p)^2k_{\text{leak}}\ge 0$. Since $a_2>0$, $a_1\ge 0$ and $a_0=0$, we have $n(s) = (a_2s^2+a_1)s^2> 0$ for all $s\neq 0$ as claimed.
Assume now that $a_0>0$.
Since $a_0a_2>0$, the Routh--Hurwitz criterion implies that we can have $n(s)=0$ for some $s\neq 0$ 
 only if  $a_1\leq 0 $ and $  a_1^2-4a_2a_0>0.$
  Because of this, we want to show that if $a_1\leq 0$, then necessarily $a_1^2-4a_2a_0\leq 0$. In fact, 
a direct computation shows that
  \begin{align*}
    a_1^2-4a_2a_0=&-2B_mB_p(B_m+B_p)(D_m+D_p)^2k_{\text{leak}}-2(B_m+B_p)(D_m+D_p)^2k_{\text{leak}}^3\\
    &\hspace{-0.8em}-(B_p+B_m)^2(B_mB_p-k_{\text{leak}}^2)^2-(D_m+D_p)^4k_{\text{leak}}^2+2a_1(D_m+D_p)^2k_{\text{leak}}\\
    &\hspace{-0.8em}+2a_1(B_mB_p+k_{\text{leak}}^2)(B_m+B_p)-4B_pB_m(D_p+D_m)^2(D_m-D_p)^2\\
    &\hspace{-0.8em}-4(B_p+B_m)\left[2B_pB_m(D_p^2+D_m^2)+(B_mD_p-B_pD_m)^2\right]k_{\text{leak}}.
  \end{align*}
  This expression shows that if $a_1\leq0$ then $a_1^2-4a_2a_0\leq0$. Hence $n(s)\neq0$ for any $s>0$. 
Since $\sigma(A_2)\subset \C_-$, we must have $d(s)\neq 0$ for $s\neq0$.
Moreover, we have $\re \mathbb{H}_2(is)\to B_p+B_m>0$ as $|s|\to \infty$.
Combining the above properties we conclude that for any $\varepsilon > 0$ there exists $c_\varepsilon > 0$ such that $\re \mathbb{H}_2(is) \geq c_\varepsilon I$ for all $s \in \mathbb{R} \setminus (-\varepsilon, \varepsilon)$. 
  On the other hand, it follows from Theorem \ref{thm:torque} that the operator
  $A_K:=A_1-B_1KC_1$ with domain $D(A_K)=D(A_1)$ where $K=I$ generates a contraction semigroup
$\mathbb T_K$ on $X_1$ and that $\norm{\mathbb T_Kz_0}^2=o(t\inv)$ as $t\to\infty$ for all $z_0\in D(A_1)$.
Theorem \ref{thm:BorTom2010} thus implies that
 $i\mathbb{R}\subset\rho(A_K)$ and $\|R(is,A_K)\|\lesssim 1+s^2$ for all $s\in\mathbb{R}$. Hence, Theorem \ref{thm:AbsCouplingResEst} 
with $E=\{s\in\mathbb R: \; |s|>\varepsilon\}$
 implies that $A_{\text{hd}}$ generates a contraction semigroup on $X_{\text{hd}}$
   such that $\{is\in i\mathbb{R}: \;|s|>\varepsilon\}\subset\rho(A_{\text{hd}})$ and 
that there exists $M_\varepsilon>0$ such that
  \begin{align}\label{xxxx}
    \|R(is,A_{\text{hd}})\|\leq M_\varepsilon (1+s^2), \quad \text{whenever }\; |s|>\varepsilon.
  \end{align}
   Next, we will show  that $0\in\rho(A_{\text{hd}})$.
The Sobolev embedding theorem implies that $A_{\text{hd}}$ has compact resolvents, and therefore it is sufficient to show that $\ker A_{\text{hd}}=\{0\}$.
Let $f=(f_1,f_2,f_3,f_4,f_5,f_6,f_7)^T\in D(A_{\text{hd}})$ be such that $A_{\text{hd}}f=0$. Then, we have
  \begin{align}
    &f_2=0\label{01*}, \qquad
    \left(EI(x)f''_1\right)''=0,
\qquad f_4=f_2'(1), \qquad f_3=f_2(1)
\\
    &\left(EIf''_1\right)'(1)=0,\label{03*}
    \qquad f_1(0)=f_1'(0)=0, \qquad f_2(0)=f_2'(0)=0,
\\
    &-{EI(1)}f''_1(1)+{B_p}f_5-(B_p+B_m)f_4-{B_m}f_6+(D_p+D_m)f_7=0\label{04*}\\
    &-{B_p}f_5+{B_p}f_4-{D_p}f_7=0, \label{05*}
    \qquad {D_m}f_7-{B_m}f_6-{B_m}f_4=0
\\
    &D_pf_5-(D_p+D_m)f_4-D_mf_6-k_{\text{leak}}f_7=0\label{07*}.
  \end{align}
  It follows from \eqref{01*} 
 that $f_3 = f_4 = 0$.  
Moreover, from \eqref{01*}, we deduce that $EI f_1''$ is a second-degree polynomial. Using \eqref{04*} and \eqref{05*}
 we obtain  
$EI(1) f_1''(1) = 0$. Together with \eqref{03*}, this implies that $f_1'' = 0$, since $EI$ is strictly positive.
  Using \eqref{03*} we get that $f_1=0$. Now, using
 \eqref{05*}
 and \eqref{07*}, 
   with $f_4=0$, we get that 
   \begin{align*}
    \begin{pmatrix}
      B_p& 0& D_p\\
      0&-B_m&D_m\\
      D_p&-D_m&-k_{\text{leak}}
    \end{pmatrix}\begin{pmatrix}
      f_5\\f_6\\f_7
    \end{pmatrix}=0.
   \end{align*}
The determinant of the matrix above is 
$B_p B_m k + B_p D_m^2 + B_m D_p^2$ and, since $B_p+B_m > 0$ and $D_p,D_m > 0$, it is nonsingular.
Hence we have $f_5 = f_6 = f_7 = 0$. This completes the proof that $0 \in \rho(A_{\text{hd}})$. 
Now by choosing $\varepsilon>0$ to be
   sufficiently small we can deduce that $i\mathbb{R}\subset \rho(A_{\text{hd}})$ and \eqref{xxxx} implies that $\|R(is,A_{\text{hd}})\|\lesssim {1+s^2}$ for $s\in \R$. 
  The claims therefore follow from the Arendt--Batty--Lyubich--Vu Theorem~\cite[Thm.~5.5.5(b)]{AreBat11book} and Theorem \ref{thm:BorTom2010}.
  \end{proof}

We note that the general approach of applying Theorem~\ref{thm:AbsCouplingResEst} and Theorem~\ref{thm:BorTom2010} to deduce the polynomial stability of the system~\eqref{eq:SCOLEHydro} is also applicable without the assumption that $J=J_G=J_T=1$. Indeed, 
if the parameters of the model $\Sigma_{\text{hd}}$ are fixed and if $A_{\text{hd}}$ can be expressed in the form~\eqref{eq:Ahd_def}, then the analysis of the polynomial stability of~\eqref{eq:SCOLEHydro} again reduces to showing that $(A_1,B_1,C_1,D_1)$ and $(A_2,B_2,C_2,D_2)$ are impedance passive linear systems and to deriving a lower bound for $\re \mathbb H_2(is)$.

We conclude this section by analysing the stability of~\eqref{eq:SCOLEHydro} in the case of more general parameters $J$, $J_G$, and $J_T$. We show that if $J\neq J_G$, then the system can be stabilised strongly with feedback.
More precisely, we assume that the external disturbance $\tau_{\text{rotor}}$ in~\eqref{eq:SCOLEHydro} is zero
and consider the stabilization of $\Sigma_{\text{hd}}$ using the control input $u(t)=\tau_{\text{load}}(t)$. In the stabilization we consider feedback from 
 the measurement
  \begin{align*}
    y(t)=-\frac{1}{J}w_{xt}(1,t)-\frac{1}{J_G}(\Omega_m(t)-w_{xt}(1,t)),
  \end{align*}
which is based on the angular velocity of the nacelle and the angular velocity of the generator motor.
With these choices~\eqref{eq:SCOLEHydro} can be written in the form%
\begin{subequations}
\label{hydraulic}
  \begin{align}
    \dot{z}(t) &= A_{\text{hd}}z(t) + B_{\text{hd}}u(t), \\
y(t) &= B_{\text{hd}}^\ast z(t),
  \end{align}
\end{subequations}
where $B_{\text{hd}}=\bigl(0,0,0,-\frac{1}{J},0,-\frac{1}{J_G},0\bigr)^T $.
The next theorem shows that the this system becomes strongly stable under negative output feedback.
 \begin{theorem}
The operator $A_{\textup{hd}}$ generates a contraction semigroup on $X_{\textup{hd}}$.
    Assume that $J\neq J_G$ and $B_m>0$. The hydraulic wind turbine model $\Sigma_{\text{hd}}$ described by  \eqref{hydraulic} is strongly stabilised by the state feedback $u=-ky$ with $k>0$.
  \end{theorem}
  \begin{proof}
 We begin by proving that $A_{\text {hd}}$ generates a contraction semigroup on $X_{\text {hd}}$.
  For any $f=(f_1,f_2,f_3,f_4,f_5,f_6,f_7)^T\in D(A_{\text{hd}})$, we have
      \begin{align}
\label{eq:diss}
\MoveEqLeft  \re \left\langle A_{\text{hd}}f,  f\right\rangle_{X_{\text{hd}}}
        =\re \left\langle A_{\text sc}(f_1,f_2,f_3,f_4)^T, (f_1,f_2,f_3,f_4)^T\right\rangle_{X_{\text{sc}}}\notag\\
        &\quad+\re \left(B_pf_5\overline{f_4}-(B_p+B_m)f_4\overline{f_4}-B_mf_6\overline{f_4}-(D_p+D_m)f_7\overline{f_4}\right)\notag\\
      &\quad+\re\left(-B_pf_5\overline{f_5}+B_5f_4\overline{f_5}-D_pf_7\overline{f_5}\right)
      +\re \left(D_mf_7\overline{f_6}-B_mf_6\overline{f_6}-B_mf_4\overline{f_6}\right)\notag\\
      &\quad+\re\left( D_pf_5\overline{f_7}-(D_p+D_m)f_4\overline{f_7}-D_mf_6\overline{f_7}-k_{\text{leak}}f_7\overline{f_7}\right)\notag\\
      &=-B_p|f_4-f_6|^2-B_m|f_6+f_4|^2-k_{\text{leak}}|f_7|^2\leq0,
      \end{align}
      where we have used the fact that the operator $A_{\text{sc}}$ in \eqref{scolegenerator} is skew-adjoint, see \cite[Prop. 1.1]{GuoIva2005}.
  Moreover, the Sobolev embedding theorem imply that the domain $ D(A_{\text{hd}})$ is compactly embedded in $ X_{\text{hd}}$. Consequently, 
      $A_\text{hd}$ has a compact resolvent and, in particular, $\sigma(A_{\text{hd}})\cap i\mathbb{R}$ is countable and $A_{\text {hd}}$ is maximally dissipative. Thus $A_{\text{hd}}$ generates a contraction semigroup
      $X_{\text{hd}}$ by the Lumer--Philips theorem.

    The proof of the strong stability of the semigroup generated by $A_{\text{hd}}-B_{\text{hd}}B_{\text{hd}}^\ast$
 is based on \cite[Thm. 14]{BatPho1990}. 
Let $\mu\in \R$ and $f = (f_1,\ldots,f_7)^T\in D(A_{\text{hd}})$ be such that $A_{\text{hd}}f=i\mu f$ and $B_{\text{hd}}^\ast f=0$.
We will show that $f=0$. 
After this, the strong stability follows immediately from~\cite[Thm. 14]{BatPho1990}.
We have $\re \iprod{\Ahd f}{f}_{\Xhd}=\re (i\mu \norm{f}_{\Xhd}^2) =0$ and thus~\eqref{eq:diss} implies that 
$f_4+f_6=0$.
Since $0=\Bhd^\ast f=-\frac{f_4}{J}-\frac{f_6}{J_G}$ with $J\neq J_G$, we must have $f_4=f_6=0$.
The second last row of the identity $\Ahd f=i\mu f$ implies that 
$f_7=0$, and subsequently the last row implies $f_5=0$.
With $f_4=f_5=f_6=f_7=0$ the identity $\Ahd f=i\mu f$ reduces to 
$A_{\text{sc}}g = i\mu g$, where $g=(f_1,f_2,f_3,0)^T\in D(A_{\text{sc}})$ and where $A_{\text{sc}}$ is as in~\eqref{scolegenerator}.
If we define $C_{\text{sc}}=(0,0,0,1)$, then $(C_{\text{sc}},A_{\text{sc}})$ is approximately observable by~\cite[Cor. 2.2]{GuoIva2005}. Since $C_{\text{sc}}g=0$, 
the approximate observability implies
that $g=0$, which further implies $f=0$.
Since $\Ahd$ has compact resolvents and since $\mu$ and $f$ were arbitrary, we have from~\cite[Thm. 14]{BatPho1990} that the semigroup generated by $\Ahd-\Bhd\Bhd^\ast$ is strongly stable.
\end{proof}

\bibliographystyle{plain}
\bibliography{references}

\end{document}